\documentclass[onefignum,onetabnum]{siamart190516}

\usepackage{lipsum}
\usepackage{amsfonts}
\usepackage{amssymb}
\usepackage{graphicx}
\usepackage{epstopdf}
\usepackage{algorithmic}
\usepackage{cancel}
\ifpdf
  \DeclareGraphicsExtensions{.eps,.pdf,.png,.jpg}
\else
  \DeclareGraphicsExtensions{.eps}
\fi

\newsiamremark{remark}{Remark}
\newsiamremark{hypothesis}{Hypothesis}
\crefname{hypothesis}{Hypothesis}{Hypotheses}
\newsiamthm{claim}{Claim}

\headers{Second-order uniformly AP space-time-ImEx schemes}{L. Reboul, T. Pichard, and M. Massot}

\title{Second-order uniformly asymptotic-preserving space-time-ImEx 
schemes for hyperbolic balance laws with stiff relaxation\thanks{\funding{The first author is funded by a joint PhD grant from the French Ministry of Defence (AID - Agence Innovation D\'efense) and the R\'egion \^Ile-de-France (DIM MathInnov).}}}

\author{Louis Reboul\thanks{CMAP, CNRS, \'Ecole polytechnique, Institut Polytechnique de Paris, Route de Saclay, 91128 Palaiseau Cedex, France 
  (\email{louis.reboul@polytechnique.edu}).}
\and Teddy Pichard\footnotemark[2]
\and Marc Massot\footnotemark[2]}

\usepackage{amsopn}

\newcommand{\StateVec}{\boldsymbol{w}}
\newcommand{\flx}{\boldsymbol{f}}
\newcommand{\St}{\boldsymbol{S}}
\newcommand{\dx}{\partial_{x}}
\newcommand{\dt}{\partial_{t}}

\ifpdf
\hypersetup{
  pdftitle={AP-schemes for stiff hyperbolic-relaxation problems},
  pdfauthor={L. Reboul, T. Pichard, and M. Massot}
}
\fi

\begin{document}
\maketitle

\begin{abstract}
We consider hyperbolic systems of conservation laws with relaxation source terms leading to a diffusive asymptotic limit under a parabolic scaling.
We introduce a new class of second-order in time and space numerical schemes, which are uniformly asymptotic preserving schemes.  
The proposed Implicit-Explicit (ImEx) approach, does not follow the usual path relying on the method of lines, either with multi-step methods or Runge-Kutta methods, or semi-discretized in time equations, but is inspired from the Lax-Wendroff approach with the proper level of implicit treatment of the source term. 
As a result, it yields a very compact stencil in space and time and we are able to rigorously show that both the second-order accuracy and the stability conditions are independent of the fast scales in the asymptotic regime, including the study of boundary conditions. 
We provide an original derivation of ${l}^2$ and ${l}^{\infty}$ stability conditions of the scheme that do not deteriorate the second order accuracy without relying on a limiter of any type in the linear case, in particular for shock solutions, and extend such results to the nonlinear case, showing the novelty of the method. The prototype system for the linear case is the hyperbolic heat equation, whereas barotropic Euler equations of gas dynamics with friction are the one for the nonlinear case. 
The method is also able to yield very accurate steady solutions in the nonlinear case when the relaxation coefficient in the source term depends on space. A thorough numerical assessment of the proposed strategy is provided by investigating smooth solutions, solutions with shocks and solutions leading to a steady state with space dependent relaxation coefficient.
\end{abstract}

\begin{keywords}
  Second order Implicit-Explicit numerical methods, hyperbolic systems of conservation laws with stiff source terms, asymptotic-preserving schemes in diffusive limit, ${l}^{\infty}$-stable and shock capturing methods
\end{keywords}

\begin{AMS}
 65M06, 65M12, 35L65, 76N15, 76M45, 82C40
\end{AMS}



\section{Introduction}\label{sec:intro}


We construct and analyze numerical schemes adapted to the hyperbolic-parabolic regimes arising for linear and non-linear hyperbolic problems with a relaxation term. Such systems can result from relaxation approximations of conservation laws~\cite{Chen-Lervemore-Liu94,Natalini96} as well as fluid limits of kinetic equations~\cite{Bardos-Golse-Levermore-I91}.
Our incentive originates in the field of plasma physics to perform high-fidelity simulations of sheaths through a new class of numerical methods~\cite{alvarezlaguna20jcp,Reboul2021VKI} for fluid models~\cite{alvarezlaguna20}. However, we provide here the key ideas in a simpler context, where numerical analysis is attainable. 

We focus on a parabolic scaling corresponding to long-term term behaviors with a stiff source term, and that can be expressed by the introduction of a small parameter $\varepsilon>0$.
Consider the system in non-dimensional form:
\begin{equation}\label{eq:scaledNLhypSys}
  \partial_t \StateVec + \frac{1}{\varepsilon}\partial_x \flx\left(\StateVec\right)=\frac{\sigma}{\varepsilon^{2}}\St\left( \StateVec\right)
\end{equation}
with $\StateVec \in\Omega\subset \mathbb{R}^{n}$ a state vector, $\flx:\Omega\mapsto \mathbb{R}^{n}$ a flux function with a Jacobian matrix $A\left(\StateVec \right)$ with a full basis of eigenvalues so that \Cref{eq:scaledNLhypSys} is hyperbolic when  $\sigma=0$. 
For the sake of introducing and analyzing a new class of schemes within the scope of this work, we consider only linear source term 
$\St\left( \StateVec \right)=B\StateVec$ with $B\in \mathcal{M}_{n}$.
The asymptotic regime of \Cref{eq:scaledNLhypSys} is defined by $\varepsilon\!\to\!0$. The limit depends on the specific flux $\flx$ and source term $\St$. In this paper, we investigate both a linear case: the classical hyperbolic heat equation (HHE), and a non-linear case: the Euler equations with friction (Euler-friction). These two models are representative of the main difficulties that arise when trying to approximate solutions of equations of the form \cref{eq:scaledNLhypSys}.
As it is well known (see e.g. \cite{Jin11,Caflisch-Jin-Russo97}), the disparity of scales renders the classical Finite-Volume methods overly diffusive and consequently poorly accurate in such regimes, whereas stability conditions become exceedingly restrictive.

Preserving an accurate description of the asymptotic phenomena with good stability properties while discretizing the original PDE has been the subject of a vast body of literature within the framework of so-called asymptotic-preserving (AP) schemes. Several clever strategies have been designed: 
combination of source and flux terms~\cite{gosse2002asymptotic,Jin-Levermore96,berthon2011asymptotic}, control of the numerical dissipation~\cite{chalons2019high}, relaxation techniques when the stiffness of the problem is carried by linear operators~\cite{Jin-Pareschi-Toscani98}, ImEx strategies in order to benefit from stability conditions close to that of fully implicit schemes but at the cost of explicit computations
~\cite{boscarino2013implicit}. Several contributions, as the present one, focus on the parabolic scaling specifically~\cite{Klar98,Jin-Levermore96,Jin-Pareschi-Toscani98,Jin-Pareschi_Toscani00}. 

We focus here on ImEx techniques, that have shown to be suitable for the design of high-order in time and space AP-methods (see \cite{albi2020implicit} and ref. therein) over a wide range of application. However, their construction usually focuses on the time discretization 
and tend to overlook the challenges linked to the spatial approximations. For instance, stability analysis is routinely performed under the hypothesis that the discrete spatial operator is exact (so-called semi-discretization in time). As a result, such methods may lack of suitable spatial discretizations and sometime exploit unnecessarily complex methods. Similarly, truncation error studies are customarily replaced by simpler procedure, such as verifying only that the limit scheme when $\varepsilon\to0$ produces an approximation of the desired order of the limit system, and potentially missing accuracy loss in intermediate regimes. Current ImEx methods decouple time and space, either using linear-multistep (LM, \cite{albi2020implicit}) or Runge-Kutta (RK, \cite{boscarino2013implicit}) methods, which require large stencils and, consequently, complex boundary and initial conditions are used. Similar conclusions can be drawn for other methods (e.g. \cite{chalons2019high}). 

The purpose of the present contribution is to provide a class of numerical schemes that are high-order in space and time, uniformly asymptotic preserving in a diffusive limit, and applicable to linear and non-linear systems of conservation laws with stiff source terms. It is designed by coupling intrinsically time and space, as in the Lax-Wendroff approach~\cite{Lax_Wendroff} for systems of conservation laws, but in the presence of source~\cite{Zhang-Tabarrok99}. Such an approach does not rely on a classical ODE approach or semi-discretized equations and it yields an ImEx strategy with a very compact stencil in space and time. It has the ability to deal properly with boundary conditions and offers the possibility to conduct fine stability and truncation error analyses. In this paper, we stick to second order accuracy and only aim at providing stability contraints, which are uniform in the small parameter $\varepsilon$. But we do not go fully implicit on the (potentially) non-linear fluxes. Thus, we do not alleviate the natural parabolic stiffness limitation of explicit methods for diffusion equations as in~\cite{boscarino2013implicit,Boscarino-Russo13,Boscarino-LeFloch-Russo14}. However, we introduce a novel implicit treatment of the source  (Reverse Runge-Kutta methods presented in \cref{supsec:RRK}) and provide an original derivation of $l^2$ and $l^{\infty}$ stability conditions of the uniformly AP second order scheme in space and time.
Thus, it yields the ability to treat shocks without resorting to limiters in the linear case. 
For the first and second order schemes applied to the HHE, we conduct a detailed truncation error and stability analysis, clarifying which spatial operator are the viable options compatible with the chosen discrete time structure. We also carefully consider the question of boundary conditions and provide numerical test-cases that involve non-trivial boundary conditions. Our method is then applied to Euler-friction.
To our knowledge, no second order method both in time and space for this set of equations was provided in the literature. We eventually present a numerical assessment of the new scheme with a thorough numerical study of smooth and shock solutions for the HHE, and of smooth, shock solutions as well as a stationary solution with variable in space relaxation coefficient, showing the ability of the method to capture accurately the steady state, even if not formally well-balanced (in the sense of \cite{bouchut2004nonlinear}).

The paper is organized as follows: \Cref{sec:models} presents the models and asymptotic limits we will investigate in the present work. \Cref{sec:Classical} introduces relevant numerical notions as well as recalls the limitations of classical finite volume approach. First and second order AP ImEx methods are introduced for the HHE and thoroughly studied in \Cref{sec:HHEord1,sec:HHEord2}. Methods are then extended to the nonlinear case of Euler-friction in \Cref{sec:EF} and verified numerically in \Cref{sec:num}. 


\section{Models}\label{sec:models}

In what follows, we introduce the HHE, a linear model on which the new methods will be derived and their numerical analysis conducted, as well as the Euler-friction equations that will stand as our prototype nonlinear model.

\subsection{Hyperbolic heat equations}\label{sbsec:HHE_model}
The system of HHE \cite{Ver58,Cat48,Cat58,Max67} reads:
\begin{subequations}\label{eq:HHE}
  \begin{align}
    \partial_t E + \frac{1}{\varepsilon}\partial_x F &=0, \label{eq:HHE_E}\\
    \partial_t F + \frac{1}{\varepsilon}\partial_x E &=-\frac{\sigma}{\varepsilon^{2}}F. \label{eq:HHE_F}
  \end{align}
\end{subequations}
We do not detail here the various related physics (propagation of heat at finite speed, low-order moment model for radiative transfer modeling \cite{buet2012design}, or electrical line transmission so-called the telegrapher's equation).

\subsubsection{Goldstein-Taylor model}\label{sbsbsec:GoldTayl_model}
An equivalent form of \Cref{eq:HHE} can be found in setting $u = E+F$, $v = E-F$ to obtain the Goldstein-Taylor model \cite{gosse2002asymptotic}:
\begin{subequations}\label{eq:GoldT}
  \begin{align}
    \partial_t u + \frac{1}{\varepsilon}\partial_x u &=-\frac{\sigma}{\varepsilon^{2}}\frac{u-v}{2}, \\
    \partial_t v - \frac{1}{\varepsilon}\partial_x v &=-\frac{\sigma}{\varepsilon^{2}}\frac{v-u}{2}.
  \end{align}
\end{subequations}
This equivalent formulation is particularly convenient to study the stability of numerical methods developed in this paper. The variables $u$ and $v$ can also be seen as the Riemann invariants of the convective part of equations \cref{eq:HHE}.

\subsubsection{Boundary conditions}\label{sbsbsec:HHE_BC}
We consider a finite domain of simulation $I = \left[ x_L,x_R \right]$, with $x_L,x_R\in \mathbb{R}$, $x_L<x_R$. The question of boundary conditions naturally arises in this setting. We will consider two types of boundary conditions:
\begin{itemize}
\item[-] \emph{Periodic} boundry conditions:
$E(x_L) = E(x_R)$, 
$F(x_L) = F(x_R)$. 
\item[-] \emph{Hybrid Dirichlet-Neumann} boundary conditions: imposing a constant Dirichlet boundary conditions for $E$,  the structure of \Cref{eq:HHE} imposes $\partial_x F = -\varepsilon \partial_t E = 0$, that is homogeneous Neumann boundary conditions on $F$.
\end{itemize}
To our knowledge this second case is seldom detailed in the literature, although it appears indirectly in some exact test-cases used, e.g. in~\cite{chalons2019high,buet2012design}.

\subsubsection{Energy dissipation}
If we consider a smooth solution of \Cref{eq:HHE}, we can obtain by multiplying \Cref{eq:HHE_E} by $E$ and \Cref{eq:HHE_F} by $F$:
\begin{equation}\label{eq:HHE_maxPrcp_der}
  \partial_t \left(E^{2} + F^{2} \right)+\frac{2}{\varepsilon} \partial_x \left( EF \right) = -\frac{2\sigma}{\varepsilon^{2}}F^{2}.
\end{equation}
Assuming periodic or hybrid homogeneous Dirichlet-Neumann boundary conditions and integrating relation \cref{eq:HHE_maxPrcp_der} over $\Omega$ we obtain that:
\begin{equation}\label{eq:ContinuousMaxPrinciple}
  \partial_{t}\int_{\Omega} \left(E^{2}+F^{2}\right)\left( t,x \right) \text{d}x \leq 0.
\end{equation}
We will expect our numerical schemes to obey a discrete equivalent of \cref{eq:ContinuousMaxPrinciple}, namely it should feature some form of $l^{2}$-stability.

\subsubsection{Asymptotic regime}\label{sbsec:HHE_ar}
A general method to study the asymptotic limit of a model of the form \cref{eq:scaledNLhypSys} is to use a Chapman-Enskog like approach, see e.g.~\cite{chalons2019high}. It consists in considering the following formal expansion of the state vector $\StateVec$:
\begin{equation}\label{eq:formalExp}
  \StateVec = \StateVec_0 + \varepsilon \StateVec_1 + \mathcal{O}\left(\varepsilon^{2}\right).
\end{equation}
Injecting this expansion into \Cref{eq:scaledNLhypSys} yields a hierarchy of equations that customarily leads to the limit equation at order zero, that is a close-form equation on $\StateVec_0$.
Using expansion \cref{eq:formalExp} into \cref{eq:HHE_F} yields $F_0 = 0$ (at order $\mathcal{O}\left( \varepsilon^{-2}\right)$) and $\sigma F_1 = \partial_x E_0$ (order $\mathcal{O}\left( \varepsilon^{-1}\right)$), which ultimately leads to the diffusive limit:
\begin{equation}\label{eq:DiffLimHHE}
  \partial_t E_0 - \partial_{x}\left(\frac{1}{\sigma} \partial_{x} E_0 \right)  = 0,
\end{equation}
when injected into \Cref{eq:HHE_E} at order $\mathcal{O}\left( \varepsilon^{0}\right)$. This relatively simple linear system already remarkably embodies the notion of asymptotic regime, in that it behaves as an hyperbolic system when $\varepsilon \sim 1$ but degenerates to a parabolic diffusive limit in the regime $\varepsilon\to 0$. As such it encompasses a significant part of the challenges that arise when designing numerical methods for systems of the form \cref{eq:scaledNLhypSys} while being linear and therefore making the study of accuracy and stability properties for newly elaborated schemes relatively amenable. 

\subsection{Euler-friction}
The second model considered in this paper is that of baro\-tropic Euler equations with high friction:
\begin{subequations}\label{eq:EF}
  \begin{align}
    \partial_t \rho + \frac{1}{\varepsilon} \partial_x \left( \rho u \right) &= 0, \label{eq:EF_rho}\\
    \partial_t \left( \rho u \right) + \frac{1}{\varepsilon} \partial_x \left( \rho u^{2} + p(\rho)\right) &= -\frac{\sigma}{\varepsilon^{2}} \rho u, \label{eq:EF_rhoU}
  \end{align}
\end{subequations}
where $\rho$ is the density of the fluid and $u$ its macroscopic velocity and where we assume that pressure verifies $p^\prime(\rho)>0$ \cite{berthon2011asymptotic}. As compared to HHE, \Cref{eq:EF} introduces all the difficulties linked with nonlinear hyperbolic systems, such as the onset of shocks in finite time from regular initial data or the necessity to preserve the invariance of the convex space of admissible states $\left\{ \left( \rho,\rho u \right)\in \mathbb{R}^{2}, \rho \geq 0 \right\}$.

\subsubsection{Linearized equations}
It is possible to linearize \cref{eq:EF} around a solution $\rho_{0},\left( \rho u \right)_{0}$, considering the perturbed solution $\rho = \rho_{0}+\tilde{\rho}$ and $\left( \rho u \right)_{0}+\widetilde{\rho u}$ with $\tilde{\rho} \ll \rho_{0}$ and $\widetilde{\rho u}\ll\left( \rho u \right)_{0}$, leading to the equations:
\begin{subequations}\label{eq:linearized_EF}
  \begin{align}
    \partial_{t} \tilde{\rho}+\frac{1}{\varepsilon}\partial_{x}\widetilde{\rho u}&=0, \\
    \partial_{t} \widetilde{\rho u}+\frac{1}{\varepsilon}\partial_x \left( 2 u_{0} \widetilde{\rho u}+\left( c^{2}_{0}-u^{2}_{0} \right) \tilde{\rho}\right)&= -\frac{\sigma}{\varepsilon^{2}}\widetilde{\rho u},
  \end{align}
\end{subequations}
where $u_{0}=\left( \rho u \right)_{0}/\rho_{0}$ and $c_{0}=p^{\prime}\left( \rho_{0} \right)$. One can notice that when $u_{0}=0$ and $c_{0}=1$ equations \cref{eq:linearized_EF} reduce to equations \cref{eq:HHE}, by setting $E = \tilde{\rho}$ and $F = \widetilde{\rho u}$.

\subsubsection{Asymptotic regime}
Using an analogous approach as the one used for the HHE~\cref{eq:HHE}, one can derive the diffusive limit when $\varepsilon\to0$, see e.g.~\cite{berthon2011asymptotic}:
\begin{equation}\label{eq:DiffLimEF}
  \partial_t \rho_{0} - \partial_x \left( \frac{1}{\sigma} \partial_x \left( p\left( \rho_{0} \right) \right) \right) = 0,
\end{equation}
which in the isothermal case and with the law of perfect gases, that is $p(\rho) = c^{2} \rho$, $c>0$, is nearly identical to the diffusive limit \cref{eq:DiffLimHHE}.


\section{Numerical basic notions and notations}\label{sec:Classical}

The new methods are derived and their properties studied on the HHE. 
Consequently in this section we introduce, in a linear framework, the notations used in the rest of this paper. 
We recall the notions of stability, truncation and consistency errors for linear numerical schemes.

\subsection{Numerical approximation}
Let $w$ be a solution of \Cref{eq:scaledNLhypSys} over $I\times \mathbb{R}_{+}$. Finite-volume methods aim at approximating the average of $w$ over the cells $C_{j} = \left[x_{j+1/2}, x_{j-1/2}  \right]$ at time $t^{n} = t_0 + n\Delta t$: 
\begin{equation}\label{eq:avrg}
  \bar{w}^{n}_{j} = \frac{1}{\left|C_j\right|} \int_{C_j} w \left(t^{n},x\right) \text{d}x,
\end{equation}
for $1\leq j \leq N$ and $n \geq 0$, where $x_{j+1/2}= j\Delta x+x_L$ with a mesh size $\Delta x = \left( x_R-x_L \right)/\left( N+1 \right)$ and a time step $\Delta t>0$ both assumed constant to alleviate notations. The approximation of $\bar{w}^{n}_{j}$ is usually denoted $w^{n}_{j}$, and the accuracy of a numerical method is evaluated via its \emph{global error}: $e^{n}_{j} = \bar{w}^{n}_{j}-w^{n}_{j}$.
A method is said to be of order $p$ in time and $q$ in space if $||e^{n}_{j}|| = \mathcal{O}\left(\Delta t^{p}+\Delta x^{q}\right)$ where $||\cdot||$ is a suitable norm depending on the expected regularity of the solutions. Here we will consider the norms $||\cdot||_{\infty}$ and $||\cdot||_2$.

\subsection{Upwind scheme}
Our prototype linear finite-volume scheme is derived using upwind fluxes and centered source terms for \cref{eq:GoldT}  \cite{buet2012design}. For HHE \cref{eq:HHE}, it reads:
\begin{subequations}\label{eq:HHE_upwind}
  \begin{align}
    \frac{E^{n+1}_{j}-E^{n}_{j}}{\Delta t} +\frac{F^{n}_{j+1}-F^{n}_{j-1}}{2\Delta x}-\frac{\Delta x}{2\varepsilon} \frac{E^{n}_{j+1}-2E^{n}_{j}+E^{n}_{j-1}}{\Delta x^{2}} &= 0, \label{eq:HHE_upwindE}\\
    \frac{F^{n+1}_{j}-F^{n}_{j}}{\Delta t} +\frac{E^{n}_{j+1}-E^{n}_{j-1}}{2\Delta x}-\frac{\Delta x}{2\varepsilon} \frac{F^{n}_{j+1}-2F^{n}_{j}+F^{n}_{j-1}}{\Delta x^{2}} &= -\frac{\sigma}{2\varepsilon^{2}} F^{n}_{j}. \label{eq:HHE_upwindF}
  \end{align}
\end{subequations}
On this simple linear model all classical finite volume approximated fluxes (upwind, Roe, Rusanov, HLL) yield the exact same scheme, hence our choice of the upwind scheme as an example and reference to be later compared with AP-schemes.

\subsection{Boundary conditions}
The computation of $E^{n+1}_{1}$, $E^{n+1}_{N}$, $F^{n+1}_{1}$ and $F^{n+1}_{N}$ in the scheme \cref{eq:HHE_upwind} involves the values $E^{n+1}_{0}$, $E^{n+1}_{N+1}$, $F^{n+1}_{0}$ qnd $F^{n+1}_{N+1}$ in ghost-cells. To impose periodic boundary conditions the natural choice is:
\begin{equation}\label{bc:periodic}
  E_{0}^{n} = E_{N}^{n}, \quad F_{0}^{n} = F_{N}^{n}, \quad E_{N+1}^{n} = E_{1}^{n}, \quad F_{N+1}^{n} = F_{1}^{n}.
\end{equation}
A second order discrete form of the hybrid Dirichlet-Neumann boundary conditions presented in \Cref{sbsbsec:HHE_BC} reads:
\begin{equation}\label{bc:hybrid}
  \frac{E_{0}^{n}+E_{1}^{n}}{2} = E_{L}, \quad \frac{F_{1}^{n}-F_{0}^{n}}{\Delta x} = 0, \quad \frac{E_{N}^{n}+E_{N+1}^{n}}{2}=E_{R}, \quad \frac{F_{N+1}^{n}-F_{N}^{n}}{\Delta x} =0,
\end{equation}
where $E_{L}$ and $E_{R}$ are the values imposed on $E$ at the left and right boundaries.

Both these sets of boundary conditions are compatible with any scheme for the hyperbolic heat equations \cref{eq:HHE} that are presented in this paper.

\subsection{Convergence}
In the linear framework, convergence is obtained through Lax theorem (see e.g. \cite{Lax_thm}) by showing that the scheme is consistent and stable.

\subsubsection{Stability}

This discretization of \cref{eq:HHE} yields the \emph{$l^{2}$-diminishing} property:
\begin{equation}\label{eq:MaxPrinciple}
  \sum_{j} \left( \left(E^{n+1}_{j}\right)^{2} + \left(F^{n+1}_{j}\right)^{2} \right)\Delta x \leq \sum_{j} \left( \left(E^{n}_{j}\right)^{2} + \left(F^{n}_{j}\right)^{2} \right)\Delta x.
\end{equation}
This automatically grants the \emph{$l^{2}$-stability}, that is:
\begin{equation*}
  \exists K>0, \forall n\geq 0, \quad\sum_{j} \left( \left(E^{n}_{j}\right)^{2} + \left(F^{n}_{j}\right)^{2} \right)\Delta x \leq K \sum_{j} \left( \left(E^{0}_{j}\right)^{2} + \left(F^{0}_{j}\right)^{2} \right)\Delta x.
\end{equation*}
\begin{theorem}\label{thm:UpStability}
The upwind scheme \cref{eq:HHE_upwind} is $l^{2}$-diminishing if the condition: 
  \begin{equation}
    \frac{\Delta t}{\varepsilon \Delta x}+\frac{\sigma\Delta t}{\varepsilon^{2}} \leq 1
  \end{equation}
is satisfied or equivalently $\Delta t \leq \Delta t_{max}$ with $\Delta t_{max} = (\varepsilon^{2}\Delta x)/(\varepsilon+\sigma\Delta x)$.
\end{theorem}
See \cite{buet2012design} for a proof of this result.
This very strict stability constraint can be alleviated a little by using an implicit source term but remains of the order $\Delta t_{max} = \mathcal{O} \left( \varepsilon\Delta x\right)$.

\subsubsection{Consistency}
Consistency is studied via truncation error analysis. The \emph{truncation error} is defined by $\varepsilon^{n}_{j} = \bar{w}^{n+1}_{j}-w^{n+1}_j$, where $w^{n+1}_{j}$ is the approximation yielded by the scheme at time $t^{n+1}$ under the assumption that $w^{n}_{j}=\bar{w}^{n}_{j}$ is exact at previous time $t^n$.
For a stable scheme, the truncation error is linked to the global error via a relation of the form $||e^{n}_{j}|| \leq C \sum_{k=0}^{n} ||\varepsilon^{k}_j||$. Consequently, we need $\varepsilon^{n}_{j} = \Delta t\mathcal{O}\left(\Delta t^{p}+\Delta x^{q} \right)$ for the scheme to be order $p$ and $q$ respectively in time and space.

An equivalent notion that allows to alleviate computations is the \emph{consistency error} $c^{n}_{j} = \varepsilon^{n}_{j}/\Delta t$. Formally, the consistency error can be seen as the error obtained by injecting the exact solution of problem \cref{eq:HHE} into the numerical scheme \cref{eq:HHE_upwind}.

Lastly, as smooth solutions are considered in consistency studies, we use the equivalent formalism of finite difference for the sake of legibility. 
Basically, we replace $\bar{w}^{n}_{j}$ by $w\left(t^{n}, x_{j}\right)$ with $x_{j} = \left( j+1/2 \right)\Delta x$ in error terms.
\begin{theorem}\label{thm:UpConsistency}
  The consistency error for the upwind scheme \cref{eq:HHE_upwind} satisfies:
  \begin{equation}
      c^{n}_{j}(E) = \mathcal{O}\left(\Delta t+\frac{\Delta x}{\varepsilon} \right), \quad
      c^{n}_{j}(F) = \mathcal{O}\left(\Delta t+\Delta x + \frac{\Delta x^{2}}{\varepsilon} \right).
  \end{equation}
\end{theorem}
Hence, the global error of the upwind scheme scales as $e^{n}_{j} = \mathcal{O} \left( \Delta t + \Delta x\left(1 + 1/ \varepsilon\right)\right)$. The proof is similar to the one proposed in \cite{buet2012design}.
We conclude with a few remarks:
\begin{itemize}
  \item[-] \Cref{thm:UpConsistency,thm:UpStability} assess the claim that classical finite volume methods are inaccurate and come at high computational cost in the diffusive regime.
  \item[-] From now on, instead of considering regimes defined by $\varepsilon \to 0$ or $\varepsilon \sim 1$, as in the continuous case, we rather consider the regimes $\varepsilon \ll \Delta x$, $\varepsilon \sim \Delta x$ and $\Delta x \ll \varepsilon$, more suitable to the discrete formalism. 
\end{itemize}


\section{First-order AP-schemes}\label{sec:HHEord1}

We first present our formalism, that paves the way for the second order and nonlinear extensions of the AP-methods, on a first order AP-scheme applied to HHE \cref{eq:HHE}. At order one, various time integration methods (Runge-Kutta, multi-step, Taylor series LW) are nearly interchangeable and consequently this formalism is similar to the one used in \cite{albi2020implicit,boscarino2013implicit}.

\subsection{Principle}
Let us write the general structure of the scheme we consider:
\begin{subequations}\label{eq:StructAP1HHE}
  \begin{align}
    \frac{E^{n+1}_{j}-E^{n}_{j}}{\Delta t}+\frac{1}{\varepsilon}\left[ \partial_x F \right]^{n+1}_{j}&=0, \label{eq:StructAP1HHE_E}\\
    \frac{F^{n+1}_{j}-F^{n}_{j}}{\Delta t}+\frac{1}{\varepsilon}\left[ \partial_x E \right]^{n+1}_{j}&=-\frac{\sigma}{\varepsilon^{2}}F^{n+1}_{j}, \label{eq:StructAP1HHE_F}
  \end{align} 
\end{subequations}
where $\left[ \partial_x E \right]^{n+1}_{j}$ and $\left[ \partial_x F \right]^{n+1}_{j}$ must be consistent approximations of $\partial_x E$ and $\partial_x F$ at time $t^{n+1}$ and position $x_{j}$. This structure encompasses an implicit version of Scheme~\cref{eq:HHE_upwind} where fluxes and source are approximated at time $t^{n+1}$ instead of $t^n$.   
The reason for approximating $\partial_x F$ at time $t^{n+1}$ is that it features some fast scales and needs to be implicit in order to obtain a CFL stability condition less restrictive than that of the explicit upwind scheme.

The main difference with other ImEx schemes found in the literature \cite{albi2020implicit,boscarino2013implicit} is that we consider also $\left[\partial_x E \right]^{n+1}_{j}$ instead of $\left[\partial_x E \right]^{n}_{j}$ in \Cref{eq:StructAP1HHE_F}. We did so for two reasons: to make the scheme more symmetric, which will make stability analysis significantly simpler, and to anticipate the fact that any part of the flux can contain fast scales in the general case \cref{eq:scaledNLhypSys}, and more particularly for \cref{eq:EF}.

A fully implicit scheme would come at great computational cost as one would have to solve linear systems. The flux terms in \cref{eq:StructAP1HHE} are defined through a formal discrete equivalent of the partial derivative with respect to space of \Cref{eq:HHE}:
\begin{subequations}\label{eq:DescreteXderHHE}
  \begin{align}
    \frac{\left[ \partial_x E\right]^{n+1}_{j}-\left[ \partial_x E \right]^{n}_{j}}{\Delta t}+\frac{1}{\varepsilon}\left[ \partial_{xx} F\right]^{n}_{j} &=0, \label{eq:DescreteXderHHE_E}\\
    \frac{\left[ \partial_x F \right]^{n+1}_{j}-\left[ \partial_x F \right]^{n}_{j}}{\Delta t}+\frac{1}{\varepsilon}\left[ \partial_{xx} E\right]^{n}_{j} &= -\frac{\sigma}{\varepsilon^{2}}\left[ \partial_x F \right]^{n+1}_{j}.\label{eq:DescreteXderHHE_F}
  \end{align}
\end{subequations}
with explicit diffusion terms and an implicit source term. The reason for this choice is in the strength of this ImEx approach: it comes at the cost of a scalar linear implicit scheme, that is almost at explicit cost, but with significantly improved stability properties.

Injecting \Cref{eq:DescreteXderHHE} into \Cref{eq:StructAP1HHE}, we obtain:
\begin{subequations}\label{eq:AP1_HHE_fullStruct}
  \begin{align}
    \frac{E^{n+1}_{j}-E^{n}_{j}}{\Delta t}+\frac{M}{\varepsilon}\left[ \partial_x F\right]^{n}_{j}-\frac{M\Delta t}{\varepsilon^{2}}\left[\partial_{xx} E \right]^{n}_{j}&=0, \\
    \frac{F^{n+1}_{j}-F^{n}_{j}}{\Delta t}+\frac{M}{\varepsilon}\left[\partial_x E \right]^{n}_{j}-\frac{M\Delta t}{\varepsilon^{2}}\left[\partial_{xx} F \right]^{n}_{j} &= -\frac{\sigma M}{\varepsilon^{2}}F^{n}_{j},
  \end{align}
\end{subequations}
where $M = 1/(1+(\sigma\Delta t)/\varepsilon^{2})$. 
This factor $M$ embodies the effect of our implicit step \cref{eq:DescreteXderHHE} at the cost of an explicit step. When $\Delta t \ll \varepsilon$, it behaves as a perturbation of one of order one in $\Delta t$: $M = 1 + \mathcal{O}\left( \Delta t \right)$. However, in regimes where $\varepsilon \ll \Delta t$, it follows a scaling $M = \mathcal{O}\left(\varepsilon^{2} \right)$ so the $M$ factor will help control terms that scale as ${\varepsilon}^{-1}$ or ${\varepsilon^{-2}}$.

We will systematically discretize the second order derivative in space terms using the classical centered approximations:
\begin{equation*}
  \left[\partial_{xx} E \right]^{n}_{j} = \frac{E^{n}_{j+1}-2E^{n}_{j}+E^{n}_{j-1}}{\Delta x^{2}}, \quad \left[ \partial_{xx} F \right]^{n}_{j} = \frac{F^{n}_{j+1}-2F^{n}_{j}+F^{n}_{j-1}}{\Delta x^{2}}.
\end{equation*}

The choice of discretization for $\left[\partial_x E \right]^{n}_{j}$ and $\left[\partial_x F \right]^{n}_{j}$ is critical to recover uniform accuracy with respect to $\varepsilon$. We can use for instance the upwind or centered discretization. Ultimately our choice lie in a trade-off between accuracy and stability.
In what follows, we discuss such options, showing that surprisingly the centered approximation is a viable option while using an upwind approximation of the fluxes, as in~\cref{eq:HHE_upwind}, is not accurate in intermediate regimes where $\Delta x \sim \varepsilon$ (this case is detailed in \Cref{supsec:HHEord1FV}).

\subsection{Centered discretization of the fluxes}\label{sbsec:HHEord1ctr}
If we use center discretization:
\begin{equation*}
  \left[\partial_x E \right]^{n}_{j} = \frac{E^{n}_{j+1}-E^{n}_{j-1}}{2\Delta x},\quad \left[\partial_x F \right]^{n}_{j} = \frac{F^{n}_{j+1}-F^{n}_{j-1}}{2\Delta x},
\end{equation*}
we obtain the following first order scheme, afterward called ImEx1-ctr scheme:
\begin{subequations}\label{eq:AP1_ctr_HHE}
  \begin{align}
    \frac{E^{n+1}_j - E^n_j}{\Delta t}+\frac{M}{\varepsilon}\frac{F^n_{j+1}-F^n_{j-1}}{2\Delta x}-\frac{M\Delta t}{\varepsilon^2}\frac{E^n_{j+1}-2E^n_{j}+E^n_{j-1}}{\Delta x^2} &=0, \\ 
    \frac{F^{n+1}_j - F^n_j}{\Delta t}+\frac{M}{\varepsilon}\frac{E^n_{j+1}-E^n_{j-1}}{2\Delta x}-\frac{M\Delta t}{\varepsilon^2}\frac{F^n_{j+1}-2F^n_{j}+F^n_{j-1}}{\Delta x^2} &=-M\frac{\sigma}{\varepsilon^2}F^n_j. 
  \end{align}
\end{subequations}

\subsubsection{$l^\infty$-stability}
\begin{theorem}\label{thm:LinftyStabAP1ctr}
  The ImEx1-ctr scheme \cref{eq:AP1_ctr_HHE} together with periodic \cref{bc:periodic} or hybrid \cref{bc:hybrid} boundary conditions is $l^{\infty}$-diminishing for the variables $u,v$, that is:
  \begin{equation}\label{eq:LinftyStabilityDef}
    \max_{1\leq j \leq N} \left(  \left|u^{n+1}_{j}\right|, \left|v^{n+1}_{j} \right| \right) \leq \max_{1\leq j \leq N} \left( \left|u^{n}_{j}\right|, \left|v^{n}_{j}\right| \right),
  \end{equation}
  under the condition:
  \begin{equation}\label{cnd:LInftyStbltyHHEord1}
    \Delta t_{min}:= \frac{\varepsilon \Delta x}{2} \leq \Delta t \leq \Delta t_{max}:= \frac{\sigma\Delta x^{2}}{4}\frac{1+\sqrt{1+2\left( \frac{4\varepsilon}{\sigma\Delta x} \right)^{2}}}{2}.
  \end{equation}
\end{theorem}
\begin{proof}
  To conduct the stability analysis, we switch to variables $u=E+F$ and $v = E-F$, for which the scheme~\cref{eq:AP1_ctr_HHE}
can be rewritten:
\begin{subequations}\label{eq:reorderedAP1_ctr_HHE_uv}
  \begin{align}
    u^{n+1}_j &= \lambda_{1}u^n_j+\lambda_{2}u^n_{j+1} +\lambda_{3}u^n_{j-1}+ \lambda_{4}v^n_j,& \lambda_{1} =& 1-\frac{2M\Delta t^2}{\varepsilon^2\Delta x^2}-\frac{M\sigma \Delta t}{2\varepsilon^2},\\
    v^{n+1}_j &= \lambda_{1}v^n_j+\lambda_{2}v^n_{j-1}+\lambda_{3}v^n_{j+1}+ \lambda_{4}u^n_j,& \lambda_{2} =& \frac{M\Delta t^2}{\varepsilon^2\Delta x^2}-\frac{M\Delta t}{2\varepsilon\Delta x},\\
    \lambda_{3} &= \frac{M\Delta t^2}{\varepsilon^2\Delta x^2}+\frac{M\Delta t}{2\varepsilon\Delta x}, &\lambda_{4} =& \frac{M\sigma \Delta t}{2\varepsilon^2}. \nonumber
  \end{align}
\end{subequations}
One notices that $\lambda_{1}+\lambda_{2}+\lambda_{3}+\lambda_{4} = 1$. Therefore we can deduce from \Cref{eq:reorderedAP1_ctr_HHE_uv} that the quantities $u^{n+1}_{j}$ and $v^{n+1}_{j}$ at time $t^{n+1}$ are convex combinations of the quantities $u^{n}_{j}$ and $v^{n}_{j}$ at time $t^{n}$ so long as we have $\lambda_{1}\geq 0$, $\lambda_{2}\geq 0$, $\lambda_{3}\geq 0$ and $\lambda_{4}\geq 0$. We systematically have $\lambda_{3}\geq 0$ and $\lambda_{4}\geq 0$.
Condition $\lambda_{3}\geq 0$ leads to $\Delta t\geq \varepsilon \Delta x / 2 = \Delta t_{min}$. 
Condition $\lambda_{1}\geq 0$ 
breaks down to the study of the domain of positivity of a second-order polynomial in $\Delta t$, ultimately yielding $\Delta t_{max}$ such as defined in \cref{cnd:LInftyStbltyHHEord1}.

Under these conditions, we can express $u^{n+1}_{j}$ and $v^{n+1}_{j}$ as convex combination of $u^{n+1}_{j}$ and $v^{n+1}_{j}$ and consequently we have the discrete maximum principle:
\begin{equation*}
  \min_{0\leq j \leq N+1} \left(u^{n}_{j},v^{n}_{j} \right) \leq \min_{1\leq j \leq N} \left(u^{n+1}_{j},v^{n+1}_{j} \right) \leq \max_{1\leq j \leq N} \left(u^{n+1}_{j},v^{n+1}_{j} \right) \leq \max_{0\leq j \leq N+1} \left(u^{n}_{j},v^{n}_{j} \right),
\end{equation*}
where we recall that indexes $j=0,N+1$ stand for the ghost cells used to enforce boundary conditions. Under periodic boundary conditions \cref{bc:periodic} we have $u^{n}_{0} = u^{n}_{N}$, $u^{n}_{N+1}=u^{n}_{1}$, $v^{n}_{0} = u^{n}_{N}$ and $v^{n}_{N+1}=u^{n}_{1}$. Meanwhile under hybrid boundary conditions \cref{bc:hybrid}, we have $u^{n}_{0} = -v^{n}_{1}$, $u^{n}_{N+1}=-v^{n}_{N}$, $v^{n}_{0} = -u^{n}_{0}$ and $v^{n}_{N+1}=-u^{n}_{N}$. In both case \cref{eq:LinftyStabilityDef} follows, which concludes the proof.
\end{proof}
\newpage
A few remarks can be drawn from the previous studies:
\begin{itemize}
  \item[-] One observes that $\Delta t_{max}$ is always strictly bigger than $\Delta t_{min}$:
  \begin{equation}
    \Delta t_{max} =\frac{\sigma\Delta x^{2}}{4}\frac{1+\sqrt{1+2\left( \frac{4\varepsilon}{\sigma\Delta x} \right)^{2}}}{2} > \frac{\sigma\Delta x^{2}}{4}\frac{\sqrt{2\left( \frac{4\varepsilon}{\sigma\Delta x} \right)^{2}}}{2} = \sqrt{2} \Delta t_{min}.
  \end{equation}
It is always possible to make the ImEx1-ctr scheme
stable  irrespectively of $\varepsilon$.
\item[-]  In the regime $\varepsilon \ll \Delta x$ we have the classical parabolic stability condition:
    \begin{equation}
    \Delta t_{max} \underset{\varepsilon\to0}{\sim} \frac{\sigma \Delta x^{2}}{4}.
  \end{equation}
    In the regime $\Delta x \ll \varepsilon$, we have a classical hyperbolic stability conditions:
  \begin{equation}
    \Delta t_{max} \underset{\Delta x\to0}{\sim} \frac{\varepsilon\Delta x}{\sqrt{2}}.
  \end{equation}
  \item[-]  The numerical diffusion in scheme \cref{eq:AP1_ctr_HHE}, essential to stability, scales as $\Delta t$. This explains the need for a lower bound in $\Delta t$ to avoid suppressing the stabilizing mechanism of the scheme. 
  Possible ways of lifting this restriction are discussed at the end of this section.
  \item[-]  The $l^{\infty}$-stability is in general stronger than the $l^{2}$-stability. The convex combination form \cref{eq:reorderedAP1_ctr_HHE_uv} ensures that the scheme does not produce any spurious numerical oscillations around discontinuities. This formulation also allows to retrieve the energy inequality \cref{eq:MaxPrinciple}.
\end{itemize}

\subsubsection{$l^2$-stability}
Under a slightly less
limiting condition, $l^{2}$-stability follows.
\begin{theorem}\label{thm:L2stbltyHHEord1}
  The ImEx1-ctr scheme \cref{eq:AP1_ctr_HHE} together with periodic boundary conditions \cref{bc:periodic} is $l^{2}$-diminishing under the condition:
  \begin{equation}\label{cnd:L2stbltyHHEord1}
    \Delta t \leq \Delta t_{max} =\frac{\sigma\Delta x^2}{4}\frac{1+\sqrt{1+\left(\frac{4\varepsilon}{\sigma\Delta x}\right)^2}}{2}.
  \end{equation}
\end{theorem}
\begin{proof}
  Let us first notice that for all $n\geq0, 1\leq j\leq N$:
  \begin{equation*}
    \left( E^{n}_{j} \right)^{2}+\left( F^{n}_{j} \right)^{2}=\frac{1}{4}\left( \left( u^{n}_{j} + v^{n}_{j}\right)^{2}+\left( u^{n}_{j}-v^{n}_{j} \right)^{2} \right)=\frac{1}{2}\left( \left( u^{n}_{j} \right)^{2}+\left( v^{n}_{j} \right)^{2} \right),
  \end{equation*}
  or in other words $||E^{n}_{j}||^{2}_{l^{2}}+||F^{n}_{j}||^{2}_{l^{2}}=\frac{1}{2}\left( ||u^{n}_{j}||^{2}_{l^{2}}+||v^{n}_{j}||^{2}_{l^{2}} \right)$ so we can show that \cref{eq:MaxPrinciple} holds for variables $u$ and $v$. Because we assume periodic boundary conditions we only need to show that the Fourier modes
  $\hat{u}^{n}_{j} = \hat{u}e^{i\left( n\omega \Delta t-jk\Delta x \right)}$ and $\hat{v}^{n}_{j} = \hat{v}e^{i\left( n\omega \Delta t-jk\Delta x \right)}$,
  with $\omega,k\in \mathbb{R}$ are not amplified in $l^{2}$-norm by the scheme. 
  We have:
\begin{gather*}
    w^{n+1}_j = Aw^n_j, \quad
    w^n_j=
    \begin{pmatrix}
    \hat{u}^n_j\\
    \hat{v}^n_j
    \end{pmatrix}, \quad
    A=
    \begin{pmatrix}
    1-a-ic-b&b\\
    b&1-a+ic+b
    \end{pmatrix},\\
    a=\frac{4M\Delta t^2}{\varepsilon^2\Delta x^2}\sin^2\left(\frac{k\Delta x}{2}\right), \quad
    c=\frac{M\Delta t}{\varepsilon\Delta x}\sin\left(k\Delta x\right), \quad
    b=\frac{\sigma M\Delta t}{2\varepsilon^2}.
\end{gather*}
We can separate the hyperbolic structure from source and diffusion terms:
\begin{equation*}
    A = \frac{1}{2}\left(A_1+A_2\right), \quad
    A_1 =
    \begin{pmatrix}
    1-2a-i2c&0\\
    0&1-2a+i2c
    \end{pmatrix},\quad
    A_2 =
    \begin{pmatrix}
    1-2b&2b\\
    2b&1-2b
    \end{pmatrix}.
\end{equation*}
Matrix $A_2$ is self-adjoint, it can be diagonalized in an orthogonal basis and consequently we have $||A_2||=\rho\left(A_2\right)=\left|1-4b\right|$. Because $b$ is non-negative it is equivalent to $-1\leq 1-4b$ or $2b\leq 1$, i.e. $M\frac{\sigma \Delta t}{\varepsilon^2}\leq 1$, which always holds by definition of $M$.

The matrix $A_1$ is diagonal so that we have $||A_1||^2=\left(1-2a\right)^2+\left( 2c \right)^2$. Setting $M^\prime=2M$ for convenience in our notations, we have:
\begin{align*}
    ||A_1||^2
    =& 1-M^\prime\frac{4\Delta t^2}{\varepsilon^2\Delta x^2}\sin^2\left(\frac{k\Delta x}{2}\right)\left(2-M^\prime\left(1-\sin^2\left(\frac{k\Delta x}{2}\right)\left(1-\frac{4\Delta t^2}{\varepsilon^2\Delta x^2}\right)\right)\right).
\end{align*}
Consequently in order to satisfy the condition $||A_1||^2\leq 1$ we simply need the condition:
\begin{equation*}
    2-M^\prime\left(1-\sin^2\left(\frac{k\Delta x}{2}\right)\left(1-\frac{4\Delta t^2}{\varepsilon^2\Delta x^2}\right)\right)\geq 0.
\end{equation*}
This condition holds if 
$(4\Delta t^2)/(\varepsilon^2\Delta x^2)\leq 1$,
so we consider the case $(4\Delta t^2)/(\varepsilon^2\Delta x^2)> 1$ in what follows.
Using the definition of $M$, this last condition reduces to finding the positive root of a second-order polynomial equation, leading to \cref{cnd:L2stbltyHHEord1},
which concludes the proof.
\end{proof}
\begin{remark}
  The first noticeable point here is that there is no $\Delta t_{min}$ for the $l^{2}$-stability. The second is that even if the two maximum time steps \cref{cnd:LInftyStbltyHHEord1} and \cref{cnd:L2stbltyHHEord1} are equivalent in the regime $\varepsilon\ll \Delta x$, they differ by a factor $\sqrt{2}$ in the regime $\Delta x\ll \varepsilon$. This is due to the fact the inequality $||A_2||\leq 1$ saturates only in the regime $\varepsilon\ll\Delta x$ and is far from optimal in the regime $\Delta x \ll \varepsilon$. The actual limit on the time step $\Delta t_{max}$ for the $l^{2}$-stability is likely to be the same as that of $l^{\infty}$-stability.
\end{remark}

\subsubsection{Accuracy}
Having shown stability, we can demonstrate that accuracy is maintained throughout every regime. It is called \emph{uniformly asymptotic preserving}.
\begin{theorem}\label{thm:accuracyHHEord1}
  Suppose that \cref{cnd:L2stbltyHHEord1} holds. Then, the ImEx1-ctr scheme \cref{eq:AP1_ctr_HHE} provides $c_j^n(E) = \mathcal{O}(\Delta x)$ and $c_j^n(F)=\mathcal{O}(\Delta x)$ independently on $\varepsilon$.
\end{theorem}
\begin{proof}
  We make all necessary assumptions on the regularity of the solution $(E,F)$ to~\cref{eq:HHE} in this proof, and by an abuse of notation, we use the generic notation $w_j^n = w(t^n,x_j)$.   
  The consistency error on $E$ reads:
  \begin{align}
    c^{n}_{j}(E)
    =& \partial_{t} E^{n}_{j}+\frac{M}{\varepsilon} \partial_x F^{n}_{j} - \frac{M \Delta t}{\varepsilon^{2}}\partial_{xx} E^{n}_{j}\nonumber\\
    &+ \mathcal{O}\left( \Delta t \|\partial_{tt} E\|_\infty \right)+\mathcal{O}\left( \frac{M\Delta x^{2}}{\varepsilon} \|\partial_{xxx} F\|_\infty \right)+\mathcal{O}\left( \frac{M \Delta t \Delta x^{2}}{\varepsilon^{2}}\|\partial_{xxxx} E\|_\infty \right). \label{eq:cerrSpace_AP1_ctr_E}
  \end{align}
  By hypothesis, $\|\partial_{xxx} F\|_\infty = \varepsilon\| \partial_{txx} E\| = \mathcal{O}\left( \varepsilon \right)$ and $\Delta t = \mathcal{O}\left( \Delta x \right)$ and consequently the first two terms in \Cref{eq:cerrSpace_AP1_ctr_E} are $\mathcal{O}(\Delta x)$. Furthermore, $M\Delta t/\varepsilon^{2}\leq 1/\sigma$ and consequently the last term is $\mathcal{O}(\Delta x^2)$. Factorizing by $M$ and using~\cref{eq:HHE_F} provides: 
  \begin{align*}
    c^{n}_{j}(E) \underset{\mathcal{O}\left( \Delta x \right)}{=} & M\left( \partial_t E^{n}_{j} + \frac{1}{\varepsilon}\left( \partial_{x} F^{n}_{j} -\frac{\Delta t}{\varepsilon} \partial_{xx} E^{n}_{j} - \frac{\sigma\Delta t}{\varepsilon}\partial_{x} F^{n}_{j} \right)  \right)\\
    \underset{\mathcal{O}\left( \Delta x \right)}{=}& M\left( \partial_{t} E^{n}_{j} - \partial_{t} E^{n+1}_{j} +\mathcal{O}\left( \frac{\Delta t^{2}}{\varepsilon} \right) \right)\underset{\mathcal{O}\left( \Delta x \right)}{=}\mathcal{O}\left( M\Delta t + \frac{M\Delta t^{2}}{\varepsilon} \right).
  \end{align*}
  where $\underset{\mathcal{O}\left( \Delta x \right)}{=}$ denotes that we omited the terms that have already been shown to be uniformly of order $\mathcal{O}\left( \Delta x \right)$ in all regimes.
  Using analogous argument, we have $M\Delta t^{2}/\varepsilon = \mathcal{O}\left( \varepsilon\Delta t/\sigma \right)$ so that the first order accuracy is achieved in all regimes.

  The consistency on $F$ reads:
  \begin{align}
    c^{n}_{j}(F)
    =&\partial_{t} F^{n}_{j} + \frac{M}{\varepsilon} \partial_{x} E^{n}_{j} -\frac{M \Delta t}{\varepsilon^{2}} \partial_{xx} F^{n}_{j} + M\frac{\sigma}{\varepsilon^2}F^n_j\nonumber\\
    &+ \mathcal{O}\left( \Delta t \| \partial_{tt} F\|_\infty \right)+\mathcal{O}\left( \frac{M \Delta x^{2}}{\varepsilon} \|\partial_{xxx} E\|_\infty \right)+\mathcal{O}\left( \frac{M\Delta t \Delta x^{2}}{\varepsilon^{2}} \|\partial_{xxxx}F\|_\infty \right).\label{eq:cerrSpace_AP1_ctr_F}
  \end{align}
  The first and last terms of \Cref{eq:cerrSpace_AP1_ctr_F} can been dealt with using the same approach as before. Regarding the middle term, we consider each regime: If $\varepsilon \ll \Delta x$, then $\Delta t \approx \Delta x^{2}$ and $M \approx \varepsilon^{2}/\Delta t$ which leads to $M\Delta x^{2}/\varepsilon \approx \varepsilon \ll \Delta x$. If $\varepsilon \approx \Delta x$, then $M \Delta x^{2}/\varepsilon \leq \Delta x^{2}/\varepsilon \approx \Delta x$. If $\varepsilon \gg \Delta x$, then $\Delta x \leq \varepsilon $ and consequently $M \Delta x^{2}/\varepsilon \leq \Delta x$. Then the consistency error reads:
  \begin{align*}
    c^{n}_{j}(F) \underset{\mathcal{O}\left( \Delta x \right)}{=} &M\left( \partial_{t} F^{n}_{j}+\frac{1}{\varepsilon}\left( \partial_{x} E^{n}_{j} -\frac{\Delta t}{\varepsilon}\partial_{xx} F^{n}_{j} \right)+\frac{\sigma}{\varepsilon^{2}} \left( F^{n}_{j}+\Delta t\partial_{t}F^{n}_{j} \right) \right)\\
    \underset{\mathcal{O}\left( \Delta x \right)}{=} & M\left( \partial_{t} F^{n}_{j} + \frac{1}{\varepsilon}\partial_{x} E^{n+1}_{j}+\frac{\sigma}{\varepsilon^{2}}F^{n+1}_{j} +\mathcal{O}\left( \frac{\Delta t^{2}}{\varepsilon} +\frac{\sigma\Delta t^{2}}{\varepsilon^{2}}\right) \right)\\
    \underset{\mathcal{O}\left( \Delta x \right)}{=} &\mathcal{O}\left( M \Delta t+ \frac{M\Delta t^{2}}{\varepsilon} +\frac{\sigma M\Delta t^{2}}{\varepsilon^{2}}\right).
  \end{align*}
  Using that $\sigma M\Delta t^{2}/\varepsilon^{2} \leq \Delta t$, we retrieve order one accuracy in all regimes.
\end{proof}
\begin{remark}
  In this proof, even though second order centered discretization of the fluxes have been used, the error term $\mathcal{O}\left( M\Delta x^{2}/\varepsilon \right)$ degenerates to order 1 in the regime $\Delta x \approx \varepsilon$. This is not an issue since we only aim at first order accuracy. This aspect will be dealt with in \Cref{sec:HHEord2} for the second order scheme.
\end{remark}

\subsection{Alternative fluxes} We have shown that scheme~\cref{eq:AP1_ctr_HHE} possesses  strong robustness property under conditions~\cref{cnd:LInftyStbltyHHEord1}; however, the presence of a lower bound $\Delta t_{min}$ to ensure that shock do not trigger spurious oscillations can be seen as too restrictive. It is possible to remove this minimal time step condition by choosing more diffusive approximations for the fluxes.
We suggest the following:
\begin{subequations}\label{eq:AP1_itr_HHE}
  \begin{align}
   \left[ \partial_x F \right]^{n+1}_{j} = F_{j+\frac{1}{2}}^{n+1} - F_{j-\frac{1}{2}}^{n+1}, \qquad F_{j+\frac{1}{2}}^{n+1} = \frac{1}{2}\left(F_j^{n+1} + F_j^{n+1} + \lambda_{j+\frac{1}{2}} (E_{j+1}^{n} - E_{j}^{n}) \right), \label{eq:flux_AP1_itr_HHE_F} \\
   \left[ \partial_x E \right]^{n+1}_{j} = E_{j+\frac{1}{2}}^{n+1} - E_{j-\frac{1}{2}}^{n+1}, \qquad E_{j+\frac{1}{2}}^{n+1} = \frac{1}{2}\left(E_j^{n+1} + E_j^{n+1} + \lambda_{j+\frac{1}{2}} (F_{j+1}^{n} - F_{j}^{n}) \right), \label{eq:flux_AP1_itr_HHE_E}
  \end{align}
where $\lambda_{j+1/2} = \max\left( \frac{|F^{n}_{j+1}|}{E^{n}_{j+1}},\frac{|F^{n}_{j}|}{E^{n}_{j}} \right)$. This eventually leads to:
  \begin{align}
    &\frac{E^{n+1}_j - E^n_j}{\Delta t}+\frac{M}{\varepsilon}\frac{F^n_{j+1}-F^n_{j-1}}{2\Delta x}-\frac{M\Delta t}{\varepsilon^2}\frac{E^n_{j+1}-2E^n_{j}+E^n_{j-1}}{\Delta x^2}\\ 
    &\qquad\qquad-\frac{M\Delta x}{2\varepsilon}\frac{\lambda_{j+1/2}\left( E^n_{j+1}-E^n_{j} \right)-\lambda_{j-1/2}\left( E^n_{j} -E^n_{j-1} \right)}{\Delta x^2}=0,\nonumber\\
    &\frac{F^{n+1}_j - F^n_j}{\Delta t}+\frac{M}{\varepsilon}\frac{E^n_{j+1}-E^n_{j-1}}{2\Delta x}-\frac{M\Delta t}{\varepsilon^2}\frac{F^n_{j+1}-2F^n_{j}+F^n_{j-1}}{\Delta x^2}\\
    &\qquad\qquad-\frac{M\Delta x}{2\varepsilon}\frac{\lambda_{j+1/2}\left( F^n_{j+1}-F^n_{j} \right)-\lambda_{j-1/2}\left( F^n_{j} -F^n_{j-1} \right)}{\Delta x^2} =-M\frac{\sigma}{\varepsilon^2}F^n_j.\nonumber
  \end{align}
\end{subequations}
The diffusion term $\frac{\lambda_{j+1/2}M\Delta x}{2}\partial_{xx}E$ is now controlled in all regimes since $\lambda_{j+1/2}$ scales as $|F|=\mathcal{O}\left( \varepsilon \right)$. 
Furthermore, as $F\sim \varepsilon \partial_x E /\sigma$, the factor $\lambda_{j+1/2}$ vanishes where the variation of $E$ are small. In such areas, the scheme \cref{eq:AP1_itr_HHE} effectively reduces to the ImEx1-ctr scheme \cref{eq:AP1_ctr_HHE}. On the contrary in the vicinity of discontinuities this scheme reduces to an upwind-like scheme (see \cref{eq:AP1_dctr_HHE}). Rigorous proofs of these assertions are significantly more involved than what has been done previously and consequently they are only verified numerically.


\section{Second-order AP-schemes}\label{sec:HHEord2}

In line with the introduced formalism, the second order scheme has a general structure of the form:
\begin{subequations}\label{eq:StructAP2HHE}
  \begin{align}
    \frac{E^{n+1}_{j}-E^{n}_{j}}{\Delta t}+\frac{1}{\varepsilon}\left[ \partial_x F \right]^{n+1/2}_{j}&=0, \label{eq:StructAP2HHE_E}\\
    \frac{F^{n+1}_{j}-F^{n}_{j}}{\Delta t}+\frac{1}{\varepsilon}\left[ \partial_x E \right]^{n+1/2}_{j}&=-\frac{\sigma}{\varepsilon^{2}}F^{n+1/2}_{j}. \label{eq:StructAP2HHE_F}
  \end{align} 
\end{subequations}

\subsection{Principle}
A naive approach would be to set:
\begin{subequations}\label{eq:DescreteXderHHEnp1half}
  \begin{align}
    \frac{\left[ \partial_{x} F \right]^{n+1/2}_{j}-\left[ \partial_{x} F \right]^{n}_{j}}{\Delta t/2} &=-\frac{1}{\varepsilon}\left[ \partial_{xx} E \right]^{n}_{j}-\frac{\sigma}{\varepsilon^{2}}\left[ \partial_{x} F \right]^{n+1/2}_{j}, \label{eq:DescreteXderHHEnp1half_dxF}\\
    \frac{\left[ \partial_{x} E \right]^{n+1/2}_{j}-\left[ \partial_{x} E \right]^{n}_{j}}{\Delta t/2} &= -\frac{1}{\varepsilon} \left[ \partial_{xx} F \right]^{n}_{j}, \label{eq:DescreteXderHHEnp1half_dxE}\\
    F^{n+1/2}_{j}&=\frac{F^{n+1}_{j}+F^{n}_{j}}{2}, \label{eq:DescreteXderHHEnp1half_F}
  \end{align}
\end{subequations}
leading to the scheme:
\begin{subequations}\label{eq:AP2naive_HHE_fullStruct}
  \begin{align}
    \frac{E^{n+1}_{j}-E^{n}_{j}}{\Delta t}+\frac{M_{1/2}}{\varepsilon}\left[ \partial_x F\right]^{n}_{j}-\frac{M_{1/2}\Delta t}{\varepsilon}\left[\partial_{xx} E \right]^{n}_{j}&=0, \\
    \frac{F^{n+1}_{j}-F^{n}_{j}}{\Delta t}+\frac{M_{1/2}}{\varepsilon}\left[\partial_x E \right]^{n}_{j}-\frac{M\Delta t}{\varepsilon}\left[\partial_{xx} F \right]^{n}_{j} &= -\frac{\sigma M_{1/2}}{\varepsilon^{2}}F^{n}_{j},
  \end{align}
\end{subequations}
with $M_{1/2}=1/(1+ \sigma\Delta t/(2\varepsilon^{2}))$. Unfortunately, scheme \cref{eq:AP2naive_HHE_fullStruct} can be shown to have stability constraints in either $l^{2}$ or $l^{\infty}$ norms that scale as $\Delta t_{max}=\mathcal{O}\left( \varepsilon \Delta t \right)$. The idea is that our current choice of source term \cref{eq:DescreteXderHHEnp1half_F} is not ``implicit enough" and we have to shift it somehow a little bit more toward time $t^{n+1}$. We introduce a Reverse Runge-Kutta methodology\footnote{Called {\sl reflected} in \cite{Butcher2016} or {\sl adjoint} in \cite{Hairer1993}, it is investigated in \cite{Kalitkin14,Skvortsov17} (See \Cref{supsec:RRK}).} in order to treat the source term ``as implicitly as possible". A brief literature review and synthetic presentation of this approach for second order can be found in \Cref{supsec:RRK}. Going for second order with nice A-stability and L-stability properties, it boils down to reach time $t^{n+1/2}$ in \cref{eq:DescreteXderHHEnp1half_F} from time $t^{n+1}$:
\begin{align}
  F^{n+1/2}_j &=F^{n+1}_j-\frac{\Delta t}{2}\left[\partial_t F\right]^{n+1}_j\nonumber\\
  &=F^{n+1}_j+\frac{\Delta t}{2\varepsilon}\left(\left[\partial_{x}E\right]^{n+1}_j+\frac{\sigma}{\varepsilon}F^{n+1}_j\right)\nonumber\\
  &=\left(1+\frac{\sigma\Delta t}{2\varepsilon^2}\right)F^{n+1}_j+\frac{\Delta t}{2\varepsilon}\label{eq:BackwardStepF}
  \left[\partial_{x}E\right]^{n}_j-\frac{\Delta t^2}{2\varepsilon^2}\left[\partial_{xx} F\right]^n_j. 
\end{align}
To the knowledge of the authors, this step differs essentially from other classes of ImEx schemes found in the literature 
\cite{albi2020implicit,boscarino2013implicit}.
It is important for stability reasons that a similar backward step be used for the fluxes, starting from $t^{n+1/2}$ and coming back to $t^{n}$ because the term $\left[\partial_{xx}E\right]^n_j$ cannot be brought to $t^{n+1/2}$ without using a term of the form $\left[\partial_{xxx}F\right]^n_j$, which would increase the stencil:
\begin{align}
  \frac{\left[\partial_x F\right]^{n+1/2}_j-\left[\partial_x F\right]^{n}_j}{\Delta t/2}&=-\frac{1}{\varepsilon}\left[\partial_{xx}E\right]^n_j-\frac{\sigma}{\varepsilon^2}\left(\left[\partial_x F\right]^{n+1/2}_j-\frac{\Delta t}{2}\left[\partial_{tx} F\right]^{n+1/2}_j\right)\nonumber\\
  =-\frac{1}{\varepsilon}&\left[\partial_{xx}E\right]^n_j-\frac{\sigma}{\varepsilon^2}\left(\left[\partial_x F\right]^{n+1/2}_j+\frac{\Delta t}{2\varepsilon}\left(\left[\partial_{xx} E\right]^n_j+\frac{\sigma}{\varepsilon}\left[\partial_x F\right]^{n+1/2}_j\right)\right)\nonumber\\
  =-\frac{1}{\varepsilon}&\left(1+\frac{\sigma \Delta t}{2\varepsilon^2}\right)\left[\partial_{xx} E\right]^n_j-\frac{\sigma }{\varepsilon^2}\left(1+\frac{\sigma\Delta t}{2\varepsilon^2}\right)\left[\partial_x F\right]^{n+1/2}_j. \label{eq:BackwardStep_dxF}
\end{align}
For $\left[ \partial_{x} E \right]^{n+1/2}_{j}$, we keep the formula \cref{eq:DescreteXderHHEnp1half_dxE}. Regarding the choice of discrete operator for the fluxes $\left[ \partial_{x} E \right]^{n}_{j}$, $\left[ \partial_{x} F \right]^{n}_{j}$ we have shown in \Cref{sec:HHEord1} that the centered discretization is a valid choice and we use it as a reference in this section.
Using these different elements we obtain the ImEx2-ctr second-order centered scheme:
\begin{subequations}\label{eq:AP2_ctr_HHE}
  \begin{alignat}{2}
    &\frac{E^{n+1}_j-E^n_j}{\Delta t}+\frac{M_1}{\varepsilon}\frac{F^n_{j+1}-F^n_{j-1}}{2\Delta x}-\frac{M_1^+\Delta t}{2\varepsilon^2}\frac{E^n_{j+1}-2E^n_j+E^n_{j-1}}{\Delta x^2}=0, \label{eq:AP2_ctr_HHE_E}\\
    &\frac{F^{n+1}_{j}-F^n_j}{\Delta t}+\frac{M_2}{\varepsilon}\frac{E^n_{j+1}-E^n_{j-1}}{2\Delta x}-\frac{M_2^+\Delta t}{2\varepsilon^2}\frac{F^n_{j+1}-2F^n_j+F^n_{j-1}}{\Delta x^2}=-\frac{\sigma M_2}{\varepsilon^2}\tilde{F}^{n}_{j}, \label{eq:AP2_ctr_HHE_F} \\
    M_1=&\frac{1}{1+\frac{\sigma \Delta t}{2\varepsilon^2}\left(1+\frac{\sigma \Delta t}{2\varepsilon^2}\right)}, \ 
    M_1^+=\frac{1+\frac{\sigma\Delta t}{2\varepsilon^2}}{1+\frac{\sigma\Delta t}{2\varepsilon^2}\left(1+\frac{\sigma \Delta t}{2\varepsilon^2}\right)}, \nonumber \\
    M_2=&\frac{1+\frac{\sigma \Delta t}{2\varepsilon^2}}{1+\frac{\sigma \Delta t}{\varepsilon^2}\left(1+\frac{\sigma \Delta t}{2\varepsilon^2}\right)}, \  
    M_2^+=\frac{1+\frac{\sigma\Delta t}{\varepsilon^2}}{1+\frac{\sigma \Delta t}{\varepsilon^2}\left(1+\frac{\sigma \Delta t}{2\varepsilon^2}\right)}, \
    \tilde{F}^{n}_{j}=\frac{F^n_{j+1}+4F^n_j+F^n_{j-1}}{6}. \nonumber 
  \end{alignat}
\end{subequations}
The choice of $\tilde{F}^{n}_{j}$ will be justified during the accuracy study. Meanwhile, we can see that factors $M_{1}$, $M_{1}^{+}$, $M_{2}$ and $M_{2}^{+}$ are essentially perturbations of factor $M$, allowing for a finer control. We emphasize that the second order method has the same overall structure as the first order one. In particular it is more compact in space than other ImEx approaches based on classical RK or LM methods \cite{albi2020implicit,Boscarino-Russo13}. We now show that it is endowed with good stability and accuracy properties in all regimes.
\subsection{$l^\infty$-stability}
\label{sbsec:HHEord2_stability}
The ImEx2-ctr scheme \cref{eq:AP2_ctr_HHE} is no longer symmetric in $E$ and $F$, we introduce the variables $\tilde{u}=\sqrt{M_2}E+\sqrt{M_1}F$ and $\tilde{v}=\sqrt{M_2}E-\sqrt{M_1}F$ that diagonalize the hyperbolic part of the scheme. We can see that in regimes where $\Delta t \ll \varepsilon$, we have $
\tilde{u}\to u$ and $\tilde{v}\to v$.
\begin{theorem}\label{thm:LinftyStabAP2ctr}
The ImEx2-ctr scheme \cref{eq:AP2_ctr_HHE} together with periodic \cref{bc:periodic} or hybrid \cref{bc:hybrid} boundary conditions is $l^{\infty}$-diminishing for the variables $\tilde{u},\tilde{v}$ provided that:
  \begin{equation}\label{cnd:LInftyStbltyHHE_AP2ctr}
    \Delta t_{min}:= \varepsilon \Delta x+\frac{\sigma\Delta x}{6} \leq \Delta t \leq \Delta t_{max}:= \frac{\sigma\Delta x^{2}}{9/2}\frac{1+\sqrt{1+\frac{9}{2}\left( \frac{2\varepsilon}{\sigma\Delta x} \right)^{2}}}{2}.
  \end{equation}
\end{theorem}
\begin{proof}
  The proof is similar to that of \cref{thm:LinftyStabAP1ctr}. Transposing equations \cref{eq:AP2_ctr_HHE} to the variables $\hat{u}$ and $\hat{v}$ and reordering the terms in the manner of \Cref{eq:reorderedAP1_ctr_HHE_uv} we obtain the stability conditions:
  \begin{subequations}
    \begin{align}
      1-\frac{M_+^+\Delta t^2}{\varepsilon^2\Delta x^2}-\frac{M_2\sigma \Delta t}{3\varepsilon^2}\geq0, \label{cnd:LInftyStbltyHHE_AP2ctr1}\\
      \frac{M_+^+\Delta t^2}{2\varepsilon^2\Delta x^2}-\frac{\Tilde{M}\Delta t}{2\varepsilon\Delta x}-\frac{M_2\sigma \Delta t}{2\times 6\varepsilon^2}\geq 0,\label{cnd:LInftyStbltyHHE_AP2ctr2}\\
      \frac{M_2\sigma \Delta t}{3\varepsilon^2}-\frac{M_-^+\Delta t^2}{\varepsilon^2\Delta x^2}\geq 0, \label{cnd:LInftyStbltyHHE_AP2ctr3}
    \end{align}
  \end{subequations}
  with $\tilde{M}=\sqrt{M_{1}M_{2}}$, $M^{+}_{\pm}=\left( M_{1}^{+}\pm M_{2}^{+} \right)/2>0$.
  If $M^+_+\Delta t^2/(\varepsilon^2\Delta x^2)\leq 2/3$ and $M_2\sigma\Delta t/(3\varepsilon^2)\leq 1/3$
  then the first condition \cref{cnd:LInftyStbltyHHE_AP2ctr1} holds. 
  By construction, the latter always holds.
  For the first, one computes:
  \begin{align}
    M_1^+&
    =\frac{1+\frac{\sigma\Delta t}{\varepsilon^2}}{\frac{3}{4}+\left(\frac{1}{2}+\frac{\sigma\Delta t}{2\varepsilon^2}\right)^2}
    \leq \frac{1+\frac{\sigma\Delta t}{\varepsilon^2}}{\left(\frac{1}{2}+\frac{\sigma\Delta t}{2\varepsilon^2}\right)^2}
    \leq \frac{4}{1+\frac{\sigma\Delta t}{\varepsilon^2}}, \label{ineq:M1p}\\
    M_2^+&
    =\frac{1+\frac{\sigma \Delta t}{\varepsilon^2}}{\frac{1}{2}+\left(\frac{1}{\sqrt{2}}+\frac{\sigma\Delta t}{\sqrt{2}\varepsilon^2}\right)^2}
    \leq \frac{1+\frac{\sigma \Delta t}{\varepsilon^2}}{\left(\frac{1}{\sqrt{2}}+\frac{\sigma\Delta t}{\sqrt{2}\varepsilon^2}\right)^2}\label{ineq:M2p}
    =
    \frac{2}{1+\frac{\sigma\Delta t}{\varepsilon^2}}.
\end{align}
  In the end we retrieve the bound $M^+_+=\frac{1}{2}\left(M_1^++M_2^+\right)\leq 3/(1+(\sigma \Delta t)/(\varepsilon^2))$. The study then reduces again to finding the root of a second-order polynomial equation and we find $\Delta t_{max}$ as defined in \cref{cnd:LInftyStbltyHHE_AP2ctr}.
  The second condition \cref{cnd:LInftyStbltyHHE_AP2ctr2} holds provided that:
  \begin{equation*}
    \Delta t\geq \Delta t_{min}:= \varepsilon \Delta x+\frac{\sigma\Delta x}{6}\geq \frac{\tilde{M}}{M^+_+}\varepsilon\Delta x+\frac{M_2}{M_+^+}\frac{\sigma\Delta x^2}{6},
  \end{equation*}
  since $\tilde{M}/M^{+}_{+}\leq 1$ and $M_{2}/M^{+}_{+}\leq 1$. Finally the last condition \cref{cnd:LInftyStbltyHHE_AP2ctr3} is equivalent to:
  \begin{equation*}
      \frac{3}{8}\frac{\frac{\sigma\Delta t}{\varepsilon^2}}{\left(1+\frac{\sigma\Delta t}{2\varepsilon^2}\right)\left(1+\frac{\sigma\Delta t}{2\varepsilon^2}\left(1+\frac{\sigma\Delta t}{2\varepsilon^2}\right)\right)}\frac{\Delta t^2}{\varepsilon^2 \Delta x^2} \leq 1.
  \end{equation*}
  This condition is verified if $\Delta t\leq \frac{2}{3}\sigma \Delta x^2$ or if $\Delta t \leq 2\sqrt{\frac{2}{3}}\varepsilon\Delta x$. One of theses two upper limits is always bigger than $\Delta t_{max}$ so it is not restrictive.
\end{proof}
\begin{remark}
  It is possible that the lower and upper bounds collide in the regimes $\Delta x \approx \varepsilon$ and $\Delta x \ll \varepsilon$. This is due to the fact that the proposed $\Delta t_{min}$ and $\Delta t_{max}$ are not actual solutions of equations associated with conditions \cref{cnd:LInftyStbltyHHE_AP2ctr1,cnd:LInftyStbltyHHE_AP2ctr2}. However numerically solving these equations, we have found that in practice we always obtain $\Delta t_{max}>\Delta t_{min}$ for the actual solutions of \cref{cnd:LInftyStbltyHHE_AP2ctr1,cnd:LInftyStbltyHHE_AP2ctr2}.
\end{remark}

\subsection{$l^2$-stability}
We also have a lower-bound-free result for the $l^{2}$-norm.
\begin{theorem}\label{thm:L2stbltyHHEord2}
The ImEx1-ctr scheme \cref{eq:AP1_ctr_HHE} together with periodic boundary conditions \cref{bc:periodic} is $l^{2}$-diminishing for variables $\tilde{u}$ and $\tilde{v}$ under the condition:
  \begin{equation}\label{cnd:L2stbltyHHEord2}
    \Delta t \leq \Delta t_{max} =\frac{\sigma \Delta x^2}{6}\frac{1+\sqrt{1+6\left(\frac{2\varepsilon}{\sigma\Delta x^2}\right)^2}}{2}.
  \end{equation}
\end{theorem}
\begin{proof}
  As we use periodic boundary conditions, we study again the Fourier modes $\hat{\tilde{u}}^{n}_{j} = \hat{\tilde{u}}e^{i\left( n\omega \Delta t-jk\Delta x \right)}$ and $\hat{\tilde{v}}^{n}_{j} = \hat{\tilde{v}}e^{i\left( n\omega \Delta t-jk\Delta x \right)}$.
  The scheme writes $w^{n+1}_j=Aw^n_j$, with:
  \begin{gather*}
    w^n_j=
    \begin{pmatrix}
      \hat{\tilde{u}}^{n}_{j}\\
      \hat{\tilde{v}}^{n}_{j}
    \end{pmatrix},\quad
    A=
    \begin{pmatrix}
      1-M_+^+a-i\tilde{M}c-M_2b&M_2b-M_-^-a\\
      M_2b-M_-^-a&1-M_+^+a+i\tilde{M}c-M_2b
    \end{pmatrix}, \\
      a=\frac{2\Delta t^2}{\varepsilon^2\Delta x^2}\sin^2\left(\frac{k\Delta x}{2}\right), \quad
      c=\frac{\Delta t}{\varepsilon\Delta x}\sin\left(k\Delta x\right), \quad
      b=\frac{2+\cos\left(k\Delta x\right)}{3}\frac{\sigma\Delta t}{2\varepsilon^2}. 
  \end{gather*}
  We split the matrix $A$ into $A=\frac{1}{2}\left(A_1+A_2\right)$ where:
  \begin{align*}
      A_1&=
      \begin{pmatrix}
      1-2M^+_+a-2i\tilde{M}c & 0\\
      0&1-2M^+_+a+2i\tilde{M}c
      \end{pmatrix},\\
      A_2&=
      \begin{pmatrix}
      1-2M_2b & 2M_2b-2M^+_-a\\
      2M_2b-2M^+_-a & 1-2M_2b
      \end{pmatrix}.
  \end{align*}
  Because $A_1$ is a diagonal matrix we have $||A_1||^2=\left(1-4aM^+_+\right)^2+\left(4\tilde{M}\right)^2$. Thus, setting $M_+^\prime = 2M_1$ and $\tilde{M}^\prime = 2\tilde{M}^\prime$, we obtain:
  \begin{align*}
      ||A_1||^2
      &=1-M^\prime_+\frac{4\Delta t^2}{\varepsilon^2\Delta x^2}\sin^2\left(\frac{k\Delta x}{2}\right)\times\\
      &\left(1-M^\prime_+\left(\left(\frac{\tilde{M}^\prime}{M^\prime_+}\right)^2-\sin^2\left(\frac{k\Delta x}{2}\right)\left(\left(\frac{\tilde{M}^\prime}{M^\prime_+}\right)^2-\frac{\Delta t^2}{\varepsilon^2\Delta x^2}\right)\right)\right).
  \end{align*}
  By construction, we have  $\tilde{M}^\prime\leq M^\prime_+$. We want the following condition to hold:
  \begin{equation*}
      1-M^\prime_+\left(\left(\frac{\tilde{M}^\prime}{M^\prime_+}\right)^2-\sin^2\left(\frac{k\Delta x}{2}\right)\left(\left(\frac{\tilde{M}^\prime}{M^\prime_+}\right)^2-\frac{\Delta t^2}{\varepsilon^2\Delta x^2}\right)\right)\geq 0. 
  \end{equation*}
  When $\frac{\Delta t^2}{\varepsilon^2\Delta x^2}\leq\frac{\tilde{M}^\prime}{M^\prime_+}\leq 1$ this condition is verified. Now we consider the case where $\frac{\Delta t^2}{\varepsilon^2\Delta x^2}\geq\frac{\tilde{M}^\prime}{M^\prime_+}$. Then the condition holds if:
  \begin{equation*}
      1+\frac{\sigma\Delta t}{\varepsilon^2}\geq 6\left(\left(\frac{\tilde{M}^\prime}{M^\prime_+}\right)^2-\left(\left(\frac{\tilde{M}^\prime}{M^\prime_+}\right)^2-\frac{\Delta t^2}{\varepsilon^2\Delta x^2}\right)\right)
  \end{equation*}
  where we have used the fact that $M_+^\prime = 2 M^+_+$ and the inequality $M^+_+\leq 3/(1+\sigma\Delta t/(\varepsilon^2))$. In the end this leads to
$1+\sigma\Delta t/(\varepsilon^2)\geq 6\Delta t^2/(\varepsilon^2\Delta x^2)$,
  which eventually leads to condition \cref{cnd:L2stbltyHHEord2}.
\end{proof}
\subsection{Accuracy}\label{sbsec:HHEord2_accuracy}
The ImEx2-ctr is uniformly AP as shown in the following theorem. The proof is similar to that of \Cref{thm:accuracyHHEord1} but it is more technical and is thus presented in \Cref{apdx:Ord2AccuracyProof} for the sake of legibility.
\begin{theorem}\label{thm:accuracyHHEord2}
  Suppose that \cref{cnd:L2stbltyHHEord2,cnd:LInftyStbltyHHE_AP2ctr} holds. Then the ImEx2-ctr scheme \cref{eq:AP2_ctr_HHE} provides $c_j^n(E) = O(\Delta x^2)$ and  $c_j^n(F)= O(\Delta x^2)$ independently of $\epsilon$.
\end{theorem}

\subsection{Alternative fluxes: MUSCL reconstruction} 
\label{sbsec:HHEord2inter} 
The ImEx2-ctr scheme \cref{eq:AP2_ctr_HHE} has a restrictive stability condition in the regimes where $\varepsilon \leq \Delta x$. This means that the interval between upper and lower bounds on $\Delta t$, to ensure that no spurious oscillations appear in the vicinity of discontinuities, almost collapses.
However the use of a MUSCL method with customary approximation of the fluxes (Rusanov, HLL or Roe) suffers accuracy loss in the regime $\varepsilon \approx \Delta x$ in a similar manner to what occurs with the first order scheme. 
We circumvent this and modify scheme~\cref{eq:AP1_itr_HHE} (which does not lose accuracy in this regime) with $\lambda_{j+\frac{1}{2}} = 0$ in~\cref{eq:flux_AP1_itr_HHE_F} (no diffusion) and $\lambda_{j+\frac{1}{2}} = 1$ with a MUSCL-like reconstruction (see for instance \cite{toro2013riemann}) in~\cref{eq:flux_AP1_itr_HHE_E}.
Since the implicit flux $F$ is re-injected in the $E$ equation, the scheme eventually achieves second order accuracy in both unknowns. The slopes in this reconstruction are limited using a minmod limiter to preserve stability. This scheme is afterward referred to as ImEx2-minimod. Numerical evidence of the relevance of these choices are presented in \Cref{sec:num}.


\section{Euler-friction}\label{sec:EF}

In this section, we adapt the schemes ImEx1-ctr \cref{eq:AP1_ctr_HHE} and ImEx2-ctr \cref{eq:AP2_ctr_HHE} to barotropic Euler-friction \cref{eq:EF} that reads \Cref{eq:scaledNLhypSys} with: 
\begin{equation}
  \StateVec=
  \begin{pmatrix}
    \rho\\
    \rho u
  \end{pmatrix},\quad
  \flx\left( \StateVec \right)=
  \begin{pmatrix}
    \rho u\\
    \rho u^{2}+p\left( \rho \right)
  \end{pmatrix},\quad
  \St\left( \StateVec \right)= B \StateVec,\quad
  B=
  \begin{pmatrix}
    0&0\\
    0&1
  \end{pmatrix}.\label{eq:EF_rewritten} 
\end{equation}
with $p^{\prime}\left( \rho \right)>0$. We no longer assume that $\sigma$ is constant, it may depend on space: $\sigma(x)$, leading to a steady solution for long times for which $\rho$ can have arbitrary complex shape unlike in the $\sigma$-constant case where it must be linear.
Furthermore, we assume that the flux is homogeneous, i.e. $\flx(\StateVec) = A(\StateVec) \StateVec$ where $A(\StateVec)$ is the Jacobian of $\flx$, and more specifically we consider $p(\rho) = C\rho^\iota$. The homogeneity of the flux, essentially depending on the choice of equation of state for the pressure term, is verified for a wide variety of models (see for instance \cite{godlewski2013numerical}), including the full set of Euler equations.

\subsection{First-order AP-scheme}\label{sbsec:EFord1}
We recall the full structure of the scheme:
\begin{equation}
  \frac{\StateVec^{n+1}_{j}-\StateVec^{n}_{j}}{\Delta t}+\frac{1}{\varepsilon}\left[ \dx \flx\left( \StateVec \right) \right]^{n+1}_{j} = -\frac{B}{\varepsilon^2} \sigma_j\StateVec_j^{n+1}.
  \label{eq:scheme_EF_generic}
\end{equation}
We need to design the term $\left[ \dx \flx\left( \StateVec \right) \right]^{n+1}_{j}$. 
Assuming sufficient regularity, we have:
\begin{equation*}
  \dt \flx \left( \StateVec \right) = A\left( \StateVec \right)\dt\StateVec = -\frac{1}{\varepsilon}A\left( \StateVec \right) \dx \flx \left( \StateVec \right) -\frac{\sigma}{\varepsilon^{2}}A\left( \StateVec \right)B\StateVec.
\end{equation*}
Formally, the identity $\left[ \flx\left( \StateVec \right) \right]^{n+1}_{j} = \left[ \flx\left( \StateVec \right) \right]^{n}_{j}+\Delta t \left[ \dt\flx\left( \StateVec \right) \right]^{n}_{j}$ is rewritten:
\begin{equation}
  \left[ \flx\left(\StateVec\right) \right]^{n+1}_{j} = \left[ \flx\left( \StateVec \right) \right]^{n}_{j} - \frac{\Delta t}{\varepsilon}\left[ \left(A(\StateVec)\dx\flx\left( \StateVec \right) \right)\right]^{n}_{j} - \frac{\Delta t\sigma_j}{\varepsilon^2}\left[ \left( A(\StateVec) B \StateVec \right)\right]_j^*, \label{eq:construction_flx_EF_imp}
\end{equation}
where the flux term is chosen explicit and the source term is chosen of the form:
\begin{align}
  &\left[ A(\StateVec) B \StateVec \right]_j^* = A(\StateVec_j^n) B \StateVec_j^*, \label{eq:def_source_*_EF}\\
   &\StateVec_j^* = \StateVec_j^n - \frac{\Delta t}{\varepsilon}\left[ \dx\flx\left( \StateVec \right)\right]^{n}_{j} - \frac{\Delta t\sigma_j}{\varepsilon^2} B \StateVec_j^* = I_{M_j} \left(\StateVec_j^n - \frac{\Delta t}{\varepsilon}\left[ \dx\flx\left( \StateVec \right)\right]^{n}_{j}\right), \nonumber
\end{align}
with $I_{M_j} = Diag(1,M_j)$ and $M_j = 1 / (1+\sigma_j\Delta t/\varepsilon^2)$ as in~\cref{eq:AP1_HHE_fullStruct}.
This corresponds to fixing the non-linear part at time $t^n$ explicit and the linear part at time $t^{n+1}$ computed using~\cref{eq:scheme_EF_generic} with explicit fluxes. All together, using that $A(\StateVec) \StateVec = \flx(\StateVec)$ and applying (formally) the space derivative to~\cref{eq:construction_flx_EF_imp} leads to:
\begin{subequations}
  \label{eq:flx_EF_imp}
  \begin{align}
    \left[ \dx \flx\left( \StateVec \right) \right]^{n+1}_{j} &=  \left[ \dx \flx_{M}\left( \StateVec \right) \right]^{n}_{j} - \frac{\Delta t}{\varepsilon}\left[ \dx \left(A_M(\StateVec)\dx\flx\left( \StateVec \right)\right) \right]^{n}_{j},  \\
    A_M(\StateVec) &= A\left( \StateVec \right)I_{M}, \qquad \flx_M(\StateVec) = A_M(\StateVec)\StateVec,
  \end{align}
  which, applied to~\cref{eq:EF_rewritten}, provides:
  \begin{align}
    \flx_{M}\left( \StateVec \right)&=
    \begin{pmatrix}
      M \rho u\\
      \left( 2M-1 \right)\rho u^{2}+p\left( \rho \right)
    \end{pmatrix},\\
    A_{M}(\StateVec)\dx \flx\left( \StateVec \right) &=
    \begin{pmatrix}
      M\dx\left( \rho u^{2}+p\left( \rho \right) \right)\\
      2uM\dx \left( \rho u^{2}+p\left( \rho \right) \right)+\left( \left( p^{\prime}\left( \rho \right) \right)^{2}-u^{2} \right)\dx (\rho u)
    \end{pmatrix}.
  \end{align}
\end{subequations}
Coming back to~\cref{eq:scheme_EF_generic}, the complete structure finally reads:
\begin{equation}
  \frac{\StateVec^{n+1}_{j} - \StateVec^{n}_{j}}{\Delta t}+\frac{1}{\varepsilon}\left[I_{M} \dx\flx_{M}\left( \StateVec \right) \right]^{n}_{j}-\frac{\Delta t}{\varepsilon^{2}}\left[ I_{M}\dx \left(A_{M}\dx \flx\left( \StateVec \right) \right) \right]^{n}_{j}=-\frac{\sigma_{j} M_{j}}{\varepsilon^{2}}B\StateVec^{n}_{j}. \label{eq:scheme_ImEx1_structure}
\end{equation}
The choices made in this construction extends the ones made in the linear case in~\cref{eq:DescreteXderHHE} to the non-linear Euler-Friction equation~\cref{eq:EF_rewritten}  while preserving a tractable formula. Indeed, using this generic construction to the linear HHE~\eqref{eq:HHE} yields exactly~\eqref{eq:DescreteXderHHE}.

\subsubsection{Spatial discretization}\label{sbsbsec:EFord1inter}
We must choose how to compute the terms $\left[I_{M} \dx\flx_{M}\left( \StateVec \right) \right]^{n}_{j}$ and $\left[ I_{M}\dx (A_{M}\dx \flx\left( \StateVec \right)) \right]^{n}_{j}$. For the fluxes we choose a modified Rusanov approach (see \cref{eq:AP1_itr_HHE}) that reads:
\begin{align}
  \left[ I_{M}\dx\flx_{M}\left( \StateVec \right) \right]^{n}_{j}=& I_{M_{j}}\left[\frac{\flx_{M_{j+1}}\left( \StateVec^{n}_{j+1}\right)-\flx_{M_{j-1}
  }\left( \StateVec^{n}_{j-1} \right)}{2\Delta x} \right. \label{eq:flux_ImEX1_ctr_EF} \\
  &\left.-\frac{\Delta x}{2}\frac{\lambda_{j+1/2}\left( \StateVec^{n}_{j+1}-\StateVec^{n}_{j} \right)-\lambda_{j-1/2}\left( \StateVec^{n}_{j}-\StateVec^{n}_{j-1} \right)}{\Delta x^{2}} \right], \nonumber
\end{align}
where $\lambda_{j+1/2}=\max\left( M_{j+1}\left|u_{j+1}\right|,M_{j}\left|u_{j}\right| \right)$. Regarding the second-order derivatives in space, a classical centered discretization is used. In this discretization step, one must choose what value to use for $M_{j+1/2}$ and $u_{j+1/2}^{n}$. Here the coefficient $M_{j+1/2}$ is obtained by replacing $\sigma_j$ by $(\sigma_j+\sigma_{j+1})/2$ in $M$ and the velocity is obtained with the approximation $u_{j+1/2}^{n} = (\sqrt{\rho^{n}_{j+1}}u^{n}_{j+1}+\sqrt{\rho^{n}_{j}}u^{n}_{j})/(\sqrt{\rho^{n}_{j+1}}+\sqrt{\rho^{n}_{j}})$.

\subsubsection{Stability and accuracy}\label{sbsbsec:EFord1stability} 
The linearized version of the Euler-friction system is (up to a linear change of variable) equivalent to the HHE. Thus, the linear stability and truncation error analyses follow from that of the previous sections. Besides, we can obtain a positivity result on 
the mass density.
\begin{theorem}
The first-order ImEx scheme \cref{eq:scheme_ImEx1_structure} with fluxes \cref{eq:flux_ImEX1_ctr_EF} for the Euler-friction equations preserves the positivity of the mass density under the condition:
  \begin{equation}\label{cnd:AP2ctrEFdt}
    \Delta t_{max}\leq \frac{\sigma_{min}\Delta x^{2}}{2}\frac{\left( 1-\frac{\varepsilon\left| u \right|^{n}_{max}}{\sigma_{min}\Delta x} \right)+\sqrt{\left( 1-\frac{\varepsilon\left| u \right|^{n}_{max}}{\sigma_{min}\Delta x} \right)^{2}+4\times\frac{2\varepsilon^{2} \left( \left( \left| u \right|^{n}_{max} \right)^{2} +c^{2}\right)}{\sigma_{max}^{2}\Delta x^{2}}}}{2},
  \end{equation}
that is for $n\geq 0$, $\left( \rho^{n}_{j}\geq 0, \quad \forall j \in\left\{ 1,...,N \right\}  \right)$ implies $\left( \rho^{n+1}_{j}\geq 0, \quad \forall j \in\left\{ 1,...,N \right\} \right)$.
\end{theorem}
\begin{proof} Compute:
  \begin{align*}
    \rho^{n+1}_j =& \left(1-\frac{\Delta t}{\varepsilon\Delta x}\frac{\lambda_{j+1/2}+\lambda_{j-1/2}}{2}-\frac{\Delta t^2}{\varepsilon^2\Delta x^2}\left(M_{j+1/2}+M_{j-1/2} \right)\left( \left( u^{n}_{j} \right)^2+c^2\right)\right)\rho^n_j\\
    &+\frac{\Delta t}{\varepsilon\Delta x}\left(\frac{M_{j+1/2}\Delta t}{\varepsilon\Delta x}\left(\left( u^{n}_{j+1} \right)^{2}+c^2\right)+\frac{\lambda_{j+1/2}}{2}-\frac{M_{j+1}u_{j+1}^{n}}{2}\right)\rho^n_{j+1}\\
    &+\frac{M\Delta t}{\varepsilon\Delta x}\left(\frac{M_{j-1/2}\Delta t}{\varepsilon\Delta x}\left(\left( u^{n}_{j-1} \right)^{2}+c^2\right)+\frac{\lambda_{j-1/2}}{2}+\frac{M_{j-1}u_{j-1}^{n}}{2}\right)\rho^n_{j-1}.
  \end{align*}
  By construction $\lambda_{j+1/2}\geq \left|M_{j+1}u_{j+1}\right|$ and $\lambda_{j-1/2}\geq \left|M_{j+1}u_{j-1}\right|$, thus the factor in front of $\rho^{n}_{j\pm1}$ are positive. The factor in front of $\rho^{n}_{j}$ if the following inequality holds:
  \begin{equation*}
    1-\frac{\Delta t}{\varepsilon\Delta x}M_{max}\left| u \right|^{n}_{max}-\frac{2\Delta t^{2}}{\varepsilon^{2}\Delta x^{2}}M_{max}\left( \left( \left| u \right|^{n}_{max} \right)^{2}+c^{2} \right)\geq 0,
  \end{equation*}
  where $M_{max} = 1/\left( 1+\sigma_{min}\Delta t/\varepsilon^{2} \right)$, $\sigma_{min} = \min_{j} \sigma_{j}$ and $\left| u \right|^{n}_{max} = \max_{j} \left| u^{n}_{j} \right|$. This leads to a second-order equation which positive solution is \cref{cnd:AP2ctrEFdt}.
\end{proof}

\subsection{Second-order AP-scheme}\label{sbsec:EFord2}
The second-order scheme is of the form:
\begin{equation}
  \frac{\StateVec^{n+1}_{j}-\StateVec^{n}_{j}}{\Delta t}+\frac{1}{\varepsilon}\left[ \dx \flx\left( \StateVec \right) \right]^{n+1/2}_{j}=-\frac{\sigma_{j}}{\varepsilon^{2}}B\StateVec^{n+1/2}_{j},
  \label{eq:2ndorder_scheme_EF}
\end{equation}
where we need to define both $\left[ \dx \flx\left( \StateVec \right) \right]^{n+1/2}_{j}$ and $\StateVec^{n+1/2}$.
Concerning the source term, we use again the Reverse Runge-Kutta method as in \cref{eq:BackwardStepF}:
\begin{align}
  \StateVec^{n+1/2}_{j}
  &=\StateVec^{n+1}_{j}+\frac{\Delta t}{2\varepsilon}\left( \left[ \dx\flx\left( \StateVec \right) \right]^{n+1}_{j}+\frac{\sigma_{j}}{\varepsilon}B\StateVec^{n+1}_{j} \right). \label{eq:2ndorder_source_EF}
\end{align}
The flux term $\left[ \dx\flx\left( \StateVec \right) \right]^{n+1}_{j}$
is given by~\cref{eq:flx_EF_imp}.
For the flux term $\left[ \dx\flx\left( \StateVec \right) \right]^{n+1/2}_{j}$, we rewrite~\cref{eq:construction_flx_EF_imp} and~\cref{eq:def_source_*_EF} replacing $\Delta t$ by $\Delta t/2$ and the term $\StateVec_j^*$ at time $t^{n+\frac{1}{2}}$ is computed by: 
\begin{align*}
  \StateVec_j^* &= \StateVec_j^n - \frac{\Delta t}{2\varepsilon}\left[ \dx\flx\left( \StateVec \right)\right]^{n}_{j} - \frac{\Delta t\sigma_j}{2\varepsilon^2} B \left(\StateVec_j^* + \frac{\Delta t}{2\varepsilon}\left[\dx \flx(\StateVec)\right]_j^{n}+\frac{\sigma_j\Delta t}{2\varepsilon^2}B\StateVec_j^*\right) \\
  &= I_{M_{1_j}} \StateVec_j^n - \frac{\Delta t}{2\varepsilon}I_{M_{1_j}^+} \left[\dx \flx(\StateVec)\right]_j^n,
\end{align*}
where $M_1$ and $M_1^+$ are defined in~\cref{eq:AP2_ctr_HHE}. Here, the source term is computed at time $t^n$ using a Reverse Runge-Kutta step from time $t^{n+1/2}$ (the $*$-terms), as it is done in~\cref{eq:BackwardStep_dxF}.  
All together, this leads to the scheme:

\begin{align*}
  \frac{\StateVec^{n+1}_{j}-\StateVec^{n}_{j}}{\Delta t}&+\frac{I_{M_{3_j}}}{\varepsilon}\bigg[\left[\dx \flx_{M_1}\left( \StateVec \right) \right]^{n}_{j} - \frac{\Delta t}{2\epsilon}\left[\dx \left(A_{M_1^+}(\StateVec)\dx \flx\left( \StateVec \right) \right)\right]^{n}_{j} \\
    &-\frac{\sigma_j\Delta t}{2\varepsilon^2}B\left(\left[\dx \flx_{M} \left( \StateVec \right) \right]^{n}_{j}-\frac{\Delta t}{\epsilon}\left[\dx \left(A_{M} \dx\flx\left( \StateVec \right) \right) \right]^{n}_{j}\right) \bigg] = -\frac{\sigma_{j}M_{3_{j}}}{\varepsilon^{2}}B\StateVec^{n}_{j},
\end{align*}
where $M_{3} = 1/(1+(\sigma\Delta t/\varepsilon^{2})(1+(\sigma\Delta t/(2\varepsilon^{2}))))$.
The two flux and diffusion terms are kept separated because the coefficient $\sigma$ \textit{a priori} depends on space and do not pass under the space derivative sign.

Regarding the spatial discretization, as in Section~\ref{sbsec:HHEord2inter}, we use a centered discretization of the diffusive terms, and Lax-Friedrichs-like fluxes~\cref{eq:flux_ImEX1_ctr_EF} with $\lambda = 0$ in $\rho$ and $\lambda = 1$ in $\rho u$ combined with a MUSCL reconstruction on $\rho u$ (in a manner similar to that of \Cref{sbsec:HHEord2inter}).
Thus we have successfully extended the ImEx2-ctr scheme \cref{eq:AP2_ctr_HHE} to a nonlinear system with complex fluxes and non-constant $\sigma$ coefficient. The uniform second order accuracy can be obtained in the manner of \Cref{thm:accuracyHHEord2}.


\section{Numerical results}\label{sec:num}
We consider three test-cases. The first one is a bounda\-ry-value problem of the HHE \cref{eq:HHE}, for which an exact smooth 
solution can be derived. This analytical test-case is used to establish the order of the methods presented in this paper. 
The second test case is a Riemann problem used to demonstrate the capacity of the methods, in both linear and nonlinear settings, to handle correctly shocks. Finally the last test-case assesses how well the methods capture steady states in the nonlinear case with non-constant $\sigma$.

\subsection{Exact test-case}\label{sbsec:ExTC}
We consider the HHE complemented with Dirichlet boun\-dary conditions $E_L$, $E_F$, imposed on the $E$ variable respectively on the left and right boundaries of the domain. It can be shown that an exact solution of this problem is:
\begin{subequations}
  \begin{align*}
    E\left( t, x \right) &= f\left( t \right) g(x)+\frac{E_R-E_L}{x_R-x_L}\left( x-x_L \right)+E_L, \\
    F\left( t, x \right) &= \varepsilon f^{\prime}(t) G(x)-\frac{\varepsilon}{\sigma}\left( E_R-E_L \right),
  \end{align*}
\end{subequations}
for $t\geq 0$, $x\in [x_L,x_R]$, where:
\begin{gather*}
  f(t) = \alpha \frac{\lambda_{+}e^{\lambda_{-}t}-\lambda_{-}e^{\lambda_{+}t}}{\lambda_{+}-\lambda_{-}}+\beta\frac{e^{\lambda_{+}t}-e^{\lambda_{-}t}}{\lambda_{+}-\lambda_{-}}, \quad \lambda_{\pm} = -\frac{\sigma}{2\varepsilon^2}\left(1\mp\sqrt{1-\left(\frac{2\pi\varepsilon}{\sigma}\right)^2}\right), \\
  g(x) = \sin\left( \pi\left( x-x_L \right) \right), \quad G(x)=\frac{\cos\left( \pi\left( x-x_L \right) \right)}{\pi}.
\end{gather*}
The parameter $\alpha$ determines the amplitude of the solution and we choose $\beta = -\frac{\pi^{2}}{\sigma}\alpha$ in order for the initial condition to satisfy $\partial_{t} f = \mathcal{O}\left( \varepsilon \right)$. In the verification that follows we use $\alpha = \sigma = 1$ for simplicity.
\begin{figure}[htbp]
  \centering
  \includegraphics[width=0.98\columnwidth]{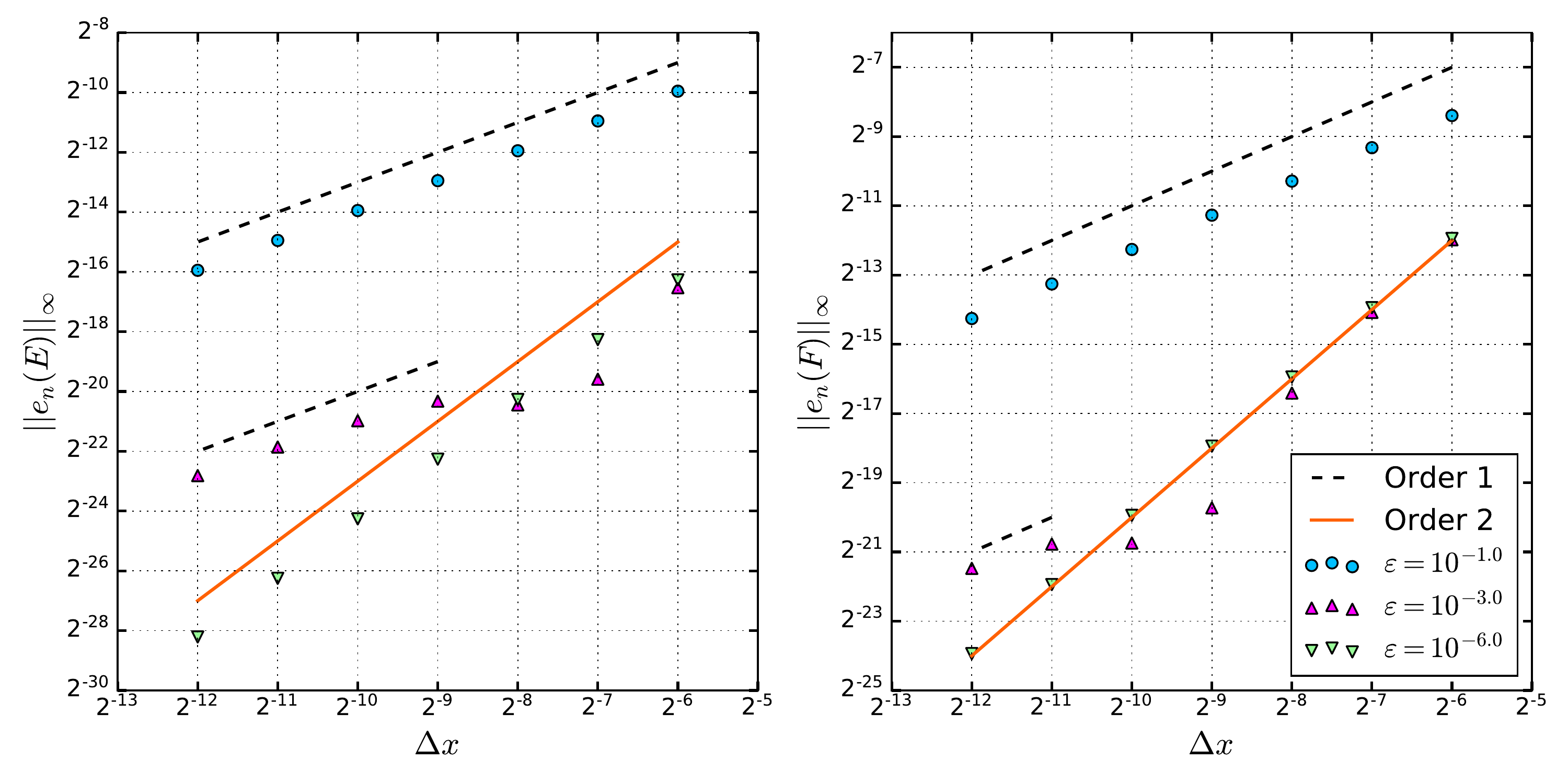}
  \caption{$l^\infty$-norm of the global errors on $E$ and $F$ of the ImEx1-ctr scheme \cref{eq:AP1_ctr_HHE} compared to the analytical solution for different values of $\varepsilon$ as a function of $\Delta x$.}
  \label{fig:OrdOfConvImEx1_ctr}
\end{figure}

In \Cref{fig:OrdOfConvImEx1_ctr}, we can observe that the ImEx1-ctr scheme \cref{eq:AP1_ctr_HHE} is of order one when we consider mesh size such that $\Delta x \leq \varepsilon$ and transition to a second order scheme in regimes where $\varepsilon \ll \Delta x$. For $\varepsilon = 10^{-3}$ the chosen range of mesh sizes $\Delta x$ encompasses both asymptotes with a transition around $\Delta x \approx \varepsilon$.
\begin{figure}[htbp]
  \centering
  \includegraphics[width=0.98\columnwidth]{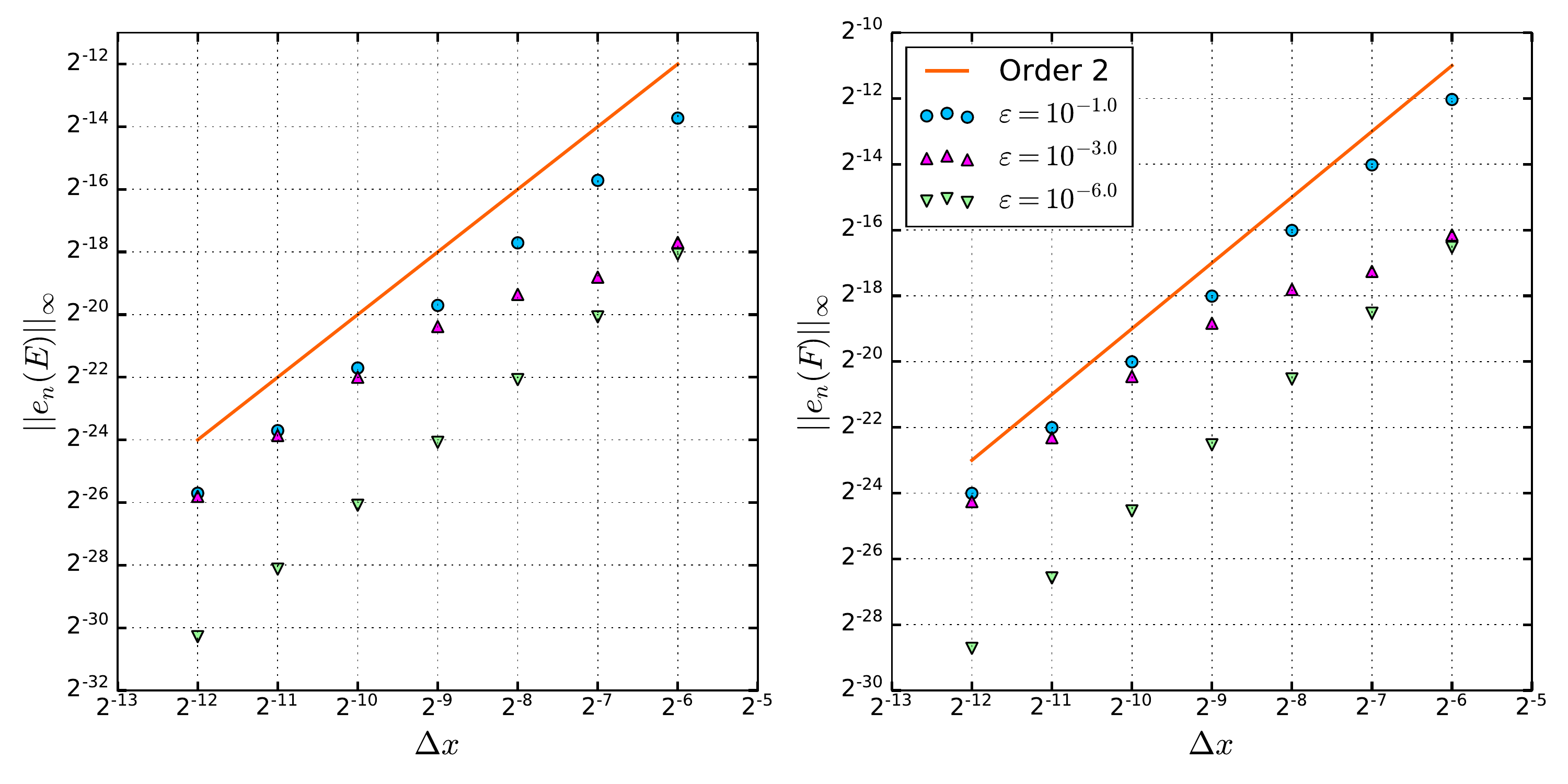}
  \caption{$l^\infty$-norm of the global errors on $E$ and $F$ of the ImEx2-ctr scheme \cref{eq:AP2_ctr_HHE} compared to the analytical solution for different values of $\varepsilon$ as a function of $\Delta x$.}
  \label{fig:OrdOfConvImEx2_ctr}
\end{figure}

We observe a similar transition, in \Cref{fig:OrdOfConvImEx2_ctr}, between two asymptotes related to the regimes $\Delta x \leq \varepsilon$ and $\Delta x\geq \varepsilon$. These asymptotes are now of slope 2, which mean the ImEx2-ctr scheme \cref{eq:AP2_ctr_HHE} guarantees second order accuracy uniformly in all regimes.
\begin{figure}[htbp]
  \centering
  \includegraphics[width=0.98\columnwidth]{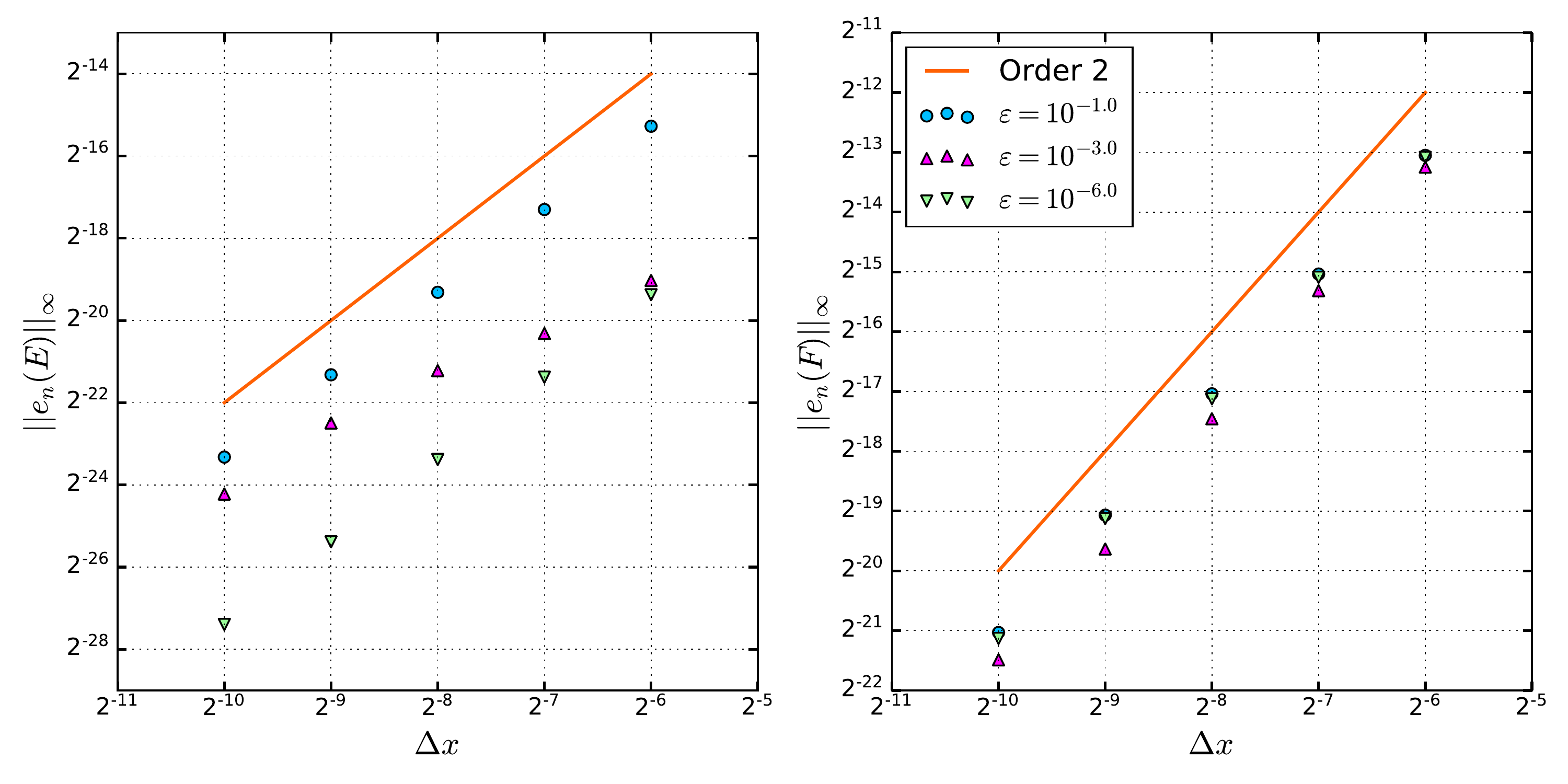}
   \caption{$l^\infty$-norm of the global errors on $E$ and $F$ of the ImEx2-minmod scheme compared to the analytical solution for different values of $\varepsilon$ as a function of $\Delta x$.}
  \label{fig:OrdOfConvImEx2_minmod}
\end{figure}

Lastly, \Cref{fig:OrdOfConvImEx2_minmod} confirms that computing the flux on variable $F$ using MUSCL reconstruction with a minmod limiter in the ImEx2-ctr scheme is still compatible with second order accuracy in all regimes.

\subsection{Riemann problem}\label{subsec:numRP}
The next test-case is that of an initial condition of the form of a Riemann problem. The linear case, which demonstrates the ability of the second-order method to be shock capturing without using slope limiters is presented in \Cref{supsec:RP}.
For the more challenging case of the nonlinear Euler-friction equations \cref{eq:EF}, this test-case reads:
  \begin{equation*}
    \rho(0,x)  = \mathbf{1}_{\left\{ x\leq \frac{x_R+x_L}{2} \right\}} \rho_L + \mathbf{1}_{\left\{ x> \frac{x_R+x_L}{2} \right\}} \rho_R, \quad
    \left( \rho u \right)(0,x ) = 0.
  \end{equation*}
\begin{figure}[htbp]
  \centering
  \includegraphics[clip, trim=0.0cm 0.5cm 0.0cm 0.0cm, width=1.\columnwidth]{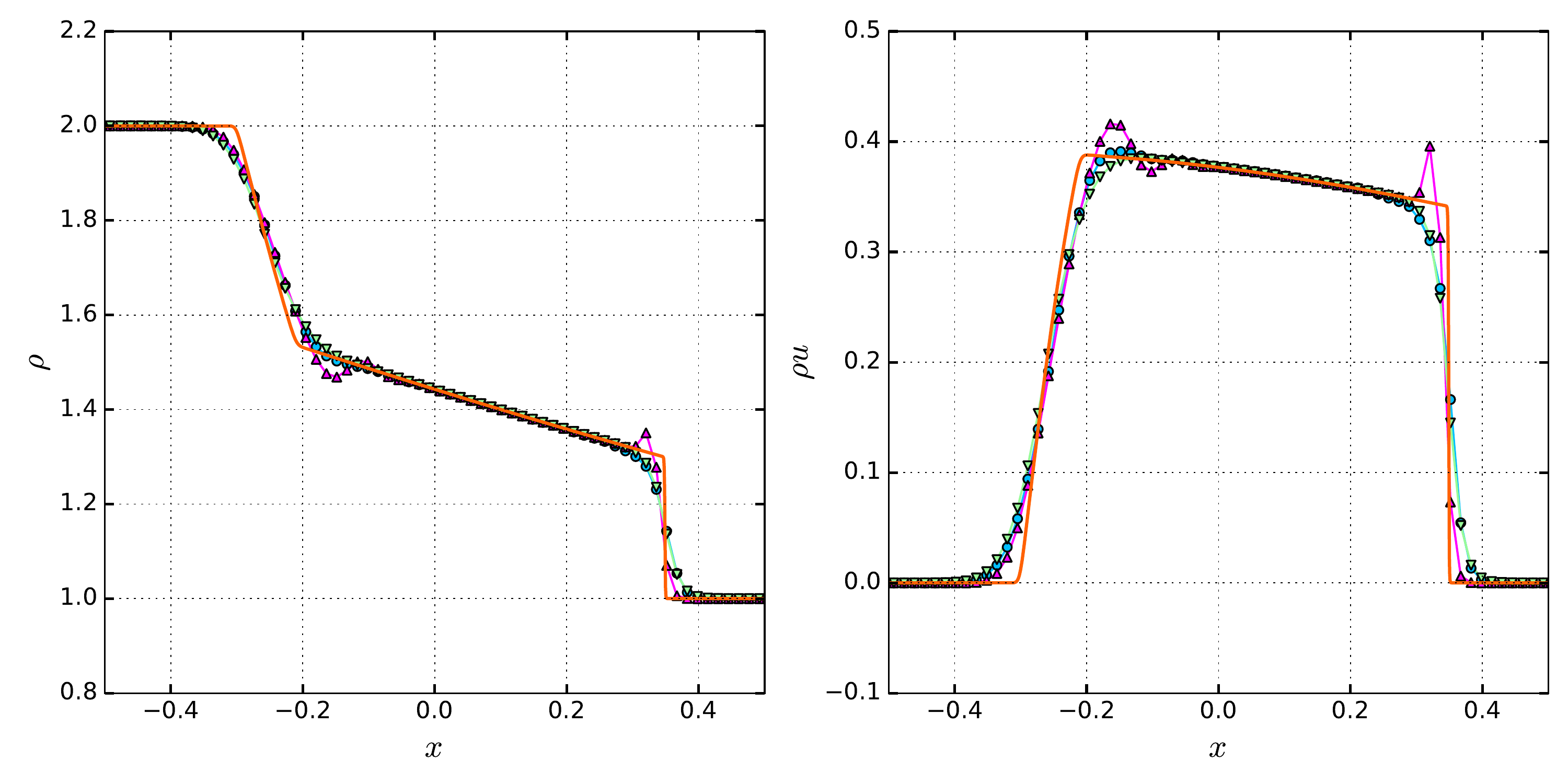}
  \caption{Approximated solutions at time $t=0.15$, with $\sigma = 1$ and $\varepsilon = 0.5$ (hyperbolic regime), computed respectively by the ImEx2-ctr method ($\triangle$), $\Delta t_{ctr} = 1.8\Delta t_{max}$, from \cref{cnd:L2stbltyHHEord2}, the ImEx2-minmod method ($\circ$), $\Delta t_{mnmd} = 1.18 \Delta t_{max}$, from \cref{cnd:L2stbltyHHEord2} and the MUSCL-Hancock method with Strang splitting, Reverse RK2 ($\triangledown$), $\Delta t_{MH} = 0.9 \min\left( 2\varepsilon^{2}/\sigma_{max},\varepsilon \Delta x/\left( u_{max}+c \right) \right)$, using $N=64$ cells, with the latter scheme used as reference (solid line) with $N=2048$ cells.}
  \label{fig:ShockEulerComp}
\end{figure}
We present the results in \Cref{fig:ShockEulerComp} for the hyperbolic regime. The hyperbolic regime is more challenging than intermediate and parabolic regimes, for which the methods can function without limiters (see \Cref{supsec:RP}). We can observe that the ImEx2-ctr scheme \cref{eq:AP2_ctr_HHE} triggers spurious oscillations whereas the ImEx2-minmod scheme essentially shows equivalent performances as the classical MUSCL-Hancock scheme coupled with Strang splitting. It must be noted that in the nonlinear case the additional numerical diffusion of the ImEx2-minmod scheme does not remove completely the lower boundary condition $\Delta t_{min}$, however the interval of possible values for the time step is large enough to be compatible with all the velocities present in the problem, unlike the ImEx2-ctr scheme.
As a result we have in the ImEx2-minmod scheme a method that incorporates enough numerical diffusion to be shock-capturing in the hyperbolic limit while retaining the asymptotic-preserving properties of the ImEx2-ctr in intermediate hyperbolic-parabolic and parabolic regimes.

\subsection{Steady state with non-constant relaxation coefficient}
The last test-case is concerned with the problem of correctly capturing steady states when the coefficient $\sigma$ is non longer assumed spatially constant. This case is particularly interesting since it allows to consider both parabolic and hyperbolic regimes in one configuration. For a specific function $\sigma(x)$, $x\in[x_L, x_R]$, the steady state solution can be easily characterized since we have for the isothermal Euler-friction equations:
\begin{equation*}
  \partial_x \left( \rho u \right) = 0,\quad \partial_x \left( \frac{\left( \rho u \right)^{2}}{\rho}+c^{2}\rho \right) =-\frac{\sigma}{\varepsilon}\rho u, \quad \rho(t,x_L) = \rho_L, \quad \rho(t,x_R) = \rho_R,
\end{equation*}
so that if we set $\rho u = \varepsilon a$ with $a\in \mathbb{R}$ we obtain:
\begin{subequations}
  \begin{align*}
    a &= \frac{\rho_L\rho_L I_\sigma}{2\varepsilon^{2}\left( \rho_L - \rho_R \right)}\left( \sqrt{1+\frac{4c^{2}\varepsilon^{2}\left( \rho_L - \rho_R \right)^{2}}{\rho_L\rho_L I_\sigma^{2}}}-1 \right), \quad
    I_\sigma = \int_{x_L}^{x_R} \sigma(x)\text{d}x, \\
    \rho(x) &= \frac{\bar{f}}{2c^{2}}\left( 1+\sqrt{1-\left( \frac{f_{min}}{\bar{f}} \right)^{2}} \right), \quad \bar{f} = \beta_\sigma f_R +\left( 1-\beta_\sigma \right)f_L, \quad f_{min} = 2\varepsilon a c
  \end{align*}
\end{subequations}
where
$f_R = \varepsilon^{2}a^{2}/\rho_R+c^{2}\rho_R$, $f_L = \varepsilon^{2}a^{2}/\rho_L+c^{2}\rho_L$, $\beta_\sigma = I_x/I_\sigma$ and $I_x = \int_{x_L}^{x}\sigma(x)\text{d}x$.
\begin{figure}[htbp]
  \centering
  \includegraphics[clip, trim=0.0cm 0.5cm 0.0cm 0.0cm, width=1.\columnwidth]{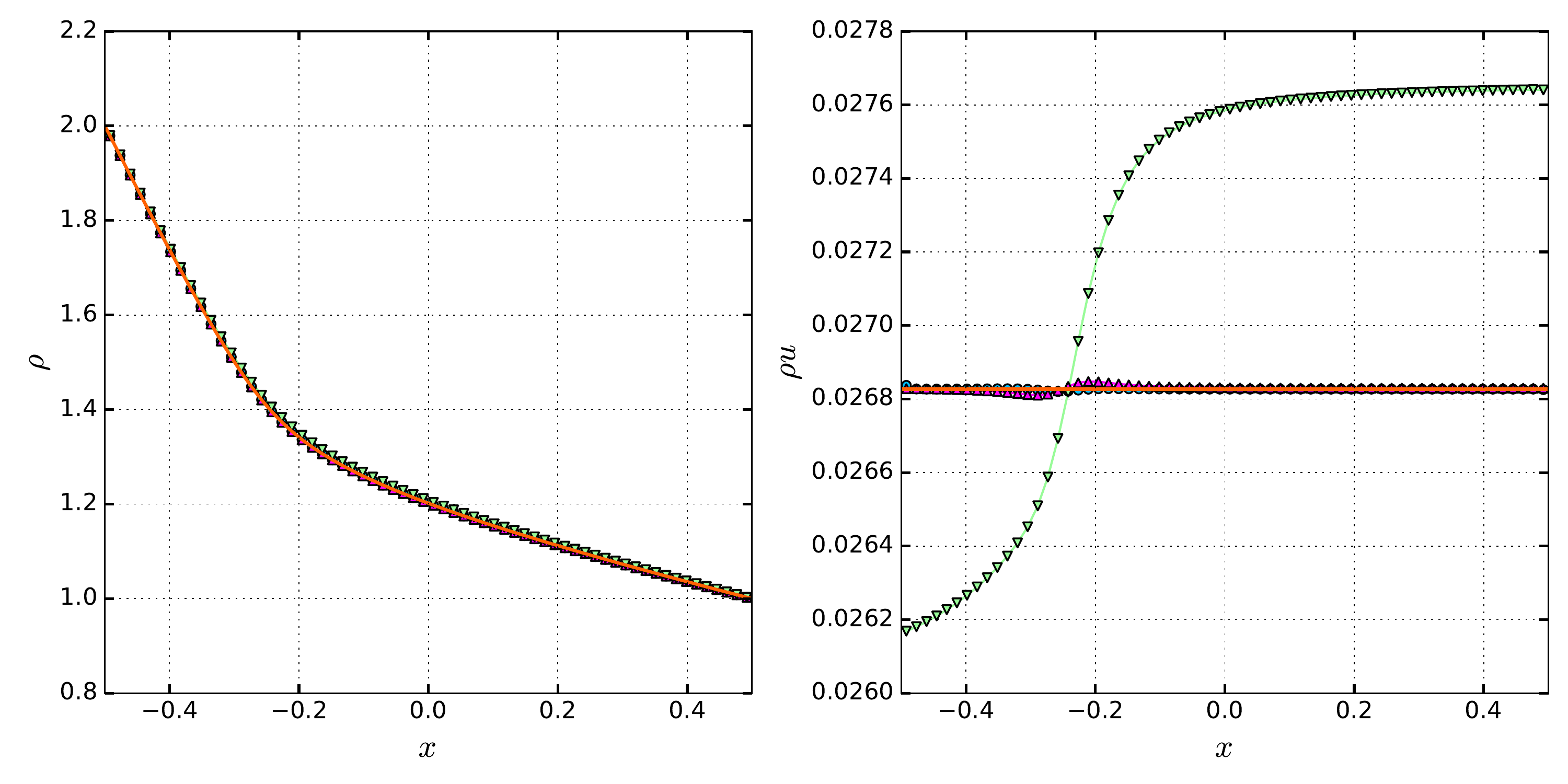}
  \caption{Density $\rho$ and momentum $\rho u$ obtained at time $t_{f} = 2$, with non-constant $\sigma$ and $\varepsilon = 10^{-2}$, respectively with the ImEx2-ctr method ($\triangle$), $\Delta t_{ctr} = 0.9\Delta t_{max}$, from \cref{cnd:L2stbltyHHEord2}, the ImEx2-minmod method ($\circ$), $\Delta t_{mnmd} = 0.9 \Delta t_{max}$, from \cref{cnd:L2stbltyHHEord2} and the MUSCL-Hancock method with Strang splitting, Reverse RK2 ($\triangledown$), $\Delta t_{MH} = 0.9 \min\left( 2\varepsilon^{2}/\sigma_{max},\varepsilon \Delta x/\left( u_{max}+c \right) \right)$, using $N=64$ cells, with the steady state solution used as reference (solid line).}
  \label{fig:SteadyStateComp}
\end{figure}

In \Cref{fig:SteadyStateComp}, we consider the case where:
\begin{equation*}
  \sigma(x) = \left( \sigma_{max}-\sigma_{min} \right)\frac{1-\frac{2}{\pi}\arctan\left( \frac{x-\frac{x_c}{2}}{\tau} \right)}{1-\frac{2}{\pi}\arctan\left( \frac{x_L-\frac{x_c}{2}}{\tau} \right)}+\sigma_{min},
\end{equation*}
where we define and chose $\sigma_{max} = 1$, $\sigma_{min} = 0.1$, $x_c = \left( x_R + x_L \right)/2$ and $\tau =$ $x_{c}/\left( 2 \tan \left( \pi/2\left( 1-\gamma \right) \right) \right)$, so that $1-\pi/2\arctan\left( x_c/\left( 2\tau \right) \right)=\gamma$, with $\gamma = 0.1$ in order to obtain a $\mathcal{C}^{\infty}$-function on $[x_L,x_R]$, going from $\sigma_{max}$ at $x_L$ to $\sigma_{min}$ at $x_R$ and reaches 90\% of its right-value at $x_c/2$.
What we observe is that although the ImEx2-ctr and ImEx2-minmod methods are not strictly well-balanced, in the sense of \cite{bouchut2004nonlinear}, especially where $\sigma^{\prime}$ is at its maximum value and at the boundaries for the minmod version, they still yield incomparably better results that the classical MUSCL-Hancock method coupled with Strang splitting.


\section{Conclusions}
\label{sec:conclusions}
We have derived two methods, ImEx2-ctr for the HHE and ImEx2-minmod for Euler-friction, which have the following properties:
\begin{itemize}
  \item[-] It is uniformly of second order accuracy in time and space in all regimes,
  \item[-] Its stability condition does not collapse with $\varepsilon$ in regimes where $\varepsilon\to0$,
  \item[-] It can be readily extended to the nonlinear case of Euler-friction equations,
  \item[-] In both linear and nonlinear setting it is shock-capturing; in the linear case and in the proper time step interval, ImEx2-ctr without slope limitation  is even $l^\infty$-stable
  \item[-] In both linear and nonlinear settings it captures accurately steady states when $\sigma$ is not constant in space.
\end{itemize}

Let us underline the novelties of the methodology. First the methods are build relying on a coupled space-time discretization in the spirit of Lax-Wendroff approach, but including an interesting  implicit treatment of the sources and fluxes involving stiff relaxation. 
In particular, a specific integration in time of the source is introduced, from the class of mono-implicit methods, with nice stability properties which outmatch the classical second order implicit method  \cite{Hairer1993,Kalitkin14,Skvortsov17}. 
This leads to fine evaluation and design of the stability properties, that are lost when time and space are decoupled as in most of the classic approaches. Besides, as opposed to the coupling of ODE solvers associated with WENO schemes for example, the proposed method is strongly compact in space and time in the spirit of a MUSCL-Hancock or Lax-Wendroff methods, while reaching interesting properties mentioned above. 

This study is the building block of a series of natural extensions. 
We presently investigate ways to ensure the shock capturing property in a less diffusive way than ImEx2-minmod using slope limiters compatible with coupling between the fluxes and the source term, that is still maintaining the space-time approach. Besides, 
we are designing extensions of the schemes presented here to two-dimensional settings, to the problem of low-Mach limit and to more complex systems of equations such as that of the Euler-Poisson equations that can be used in plasmas physics.
The ideas developed in this contribution are extended to plasma physics in a forthcoming paper and can also be used for the hyperbolic framework encountered in kinetic-fluid limits \cite{Reboul23}.


\appendix

\section{Proof of \Cref{thm:accuracyHHEord2}}\label{apdx:Ord2AccuracyProof}
  We start with the consistency error on $F$. Isolating the spatial error term reads:
  \begin{align*}
    c_j^n(F)
    =&\partial_t F^{n+1/2}_j+\frac{M_2}{\varepsilon}\partial_x E^n_j-\frac{M_2^+\Delta t}{2\varepsilon^2}\partial_{xx}F^n_j+\frac{\sigma M_2}{\varepsilon^2}F^n_j\\
    &+\mathcal{O}\left(\Delta t^2\right)+\mathcal{O}\left(\frac{M_2\Delta x^2}{6\varepsilon}\left(\partial_{xxx}E^n_j+\frac{\sigma}{\varepsilon}\partial_{xx}F^n_j\right)\right)+\mathcal{O}\left(\frac{M_2^+\Delta t}{2\varepsilon^2}\Delta x^2\right).
  \end{align*}
  Starting near the equilibrium manifold provides $\frac{1}{\varepsilon}\left(\partial_{xxx}E^n_j+\frac{\sigma}{\varepsilon}\partial_{xx}F^n_j\right)=\mathcal{O}\left(1\right)$. This legitimizes our choice of discretization for the source term.
  As $M_2^+\leq 2M$ (see~\cref{ineq:M2p}), these consistency error terms are controlled as in the proof of \Cref{thm:accuracyHHEord1}, via $M\Delta t/\varepsilon^{2}\leq 1$. Now we must compute the time-error generated in the term:
  \begin{align*}
      \frac{M_2}{\varepsilon}&\partial_x E^n_j-\frac{M_2^+\Delta t}{2\varepsilon^2}\partial_{xx}F^n_j\\
      &=\frac{1}{\varepsilon\left(1+\frac{\sigma\Delta t}{\varepsilon^2}\left(1+\frac{\sigma\Delta t}{2\varepsilon^2}\right)\right)}\left(\left(1+\frac{\sigma\Delta t}{2\varepsilon^2}\right)\partial_xE^n_j-\frac{\Delta t}{2\varepsilon}\left(1+\frac{\sigma \Delta t}{\varepsilon^2}\right)\partial_{xx}F^n_j\right)\\
      &=\frac{1}{\varepsilon\left(1+\frac{\sigma\Delta t}{\varepsilon^2}\left(1+\frac{\sigma\Delta t}{2\varepsilon^2}\right)\right)}\left(\partial_x E^{n+1/2}_j+\mathcal{O}\left(\Delta t^2\right)+\frac{\sigma \Delta t}{2\varepsilon^2}\partial_x E^{n+1}_j +\mathcal{O}\left(\frac{\Delta t^3}{\varepsilon^2}\right)\right).
  \end{align*}
  This inequality produces two new error terms that must be controled. Noticing that $x\mapsto x^{2}/\left( 1+x\left( 1+x/2 \right) \right)$ is strictly increasing on $\mathbb{R}_{+}$ and using the hypotheis $\Delta t = \mathcal{O}\left( \Delta x \left( \sigma \Delta x+\sqrt{\left( \sigma\Delta x \right)^{2}+\varepsilon^{2}} \right) \right)$, then:
  \begin{align*}  
     &\frac{\Delta t^{2}}{\varepsilon\left(1+\frac{\sigma\Delta t}{\varepsilon^2}\left(1+\frac{\sigma\Delta t}{2\varepsilon^2}\right)\right)} =\frac{\varepsilon^{3}\Delta t^{2}}{\varepsilon^{2}\left( \varepsilon^{2}+\sigma\Delta t \right)+\left( \sigma\Delta t \right)^{2}/2}\\
    &\qquad\qquad=\mathcal{O}\left( \varepsilon\Delta x^{2}\frac{\varepsilon^{2}\left( \sigma \Delta x+\sqrt{\left( \sigma\Delta x \right)^{2}+\varepsilon^{2}} \right)^{2}}{\frac{\varepsilon^{2}}{2}\left( \sigma \Delta x+\sqrt{\left( \sigma\Delta x \right)^{2}+\varepsilon^{2}} \right)^{2}+\frac{\varepsilon^{4}}{2}+\left( \varepsilon\Delta t \right)^{2}/2} \right)
    =\mathcal{O}\left( \varepsilon\Delta x^{2} \right).
  \end{align*}
  Similarly, $x\mapsto x^{3}/\left( 1+x\left( 1+x/2 \right) \right)$ is strictly increasing on $\mathbb{R}_{+}$ and using that $\Delta t = \mathcal{O}\left( \Delta x \sqrt{\left( \sigma \Delta x \right)^{2}+\varepsilon^{2}} \right)$, then:
  \begin{align*}
    \frac{\Delta t^3}{\varepsilon^3\left(1+\frac{\sigma\Delta t}{\varepsilon^2}\left(1+\frac{\sigma\Delta t}{2\varepsilon^2}\right)\right)}
    &=\mathcal{O}\left(\varepsilon\Delta x^2\frac{\Delta x\left(\Delta x^2+\varepsilon^2\right)\sqrt{\Delta x^2+\varepsilon^2}}{\varepsilon^4+\varepsilon^2\Delta x \sqrt{\Delta x^2+\varepsilon^2}+\Delta x^2\left(\Delta x^2+\varepsilon^2\right)}\right)\\
    &= \mathcal{O}\left(\varepsilon\Delta x^2\right),
  \end{align*}
  using $\sqrt{\Delta x^2+\varepsilon^2}\le \Delta x$ in the numerator and observing that the denominator is smaller than $\Delta x^2\left(\Delta x^2+\varepsilon^2\right)$.
Coming back to $c_j^n(F)$, this provides:
  \begin{align*}
  c_j^n(F) =&\frac{C_j^N(F)}{1+\frac{\sigma\Delta t}{\varepsilon^2}\left(1+\frac{\sigma \Delta t}{2\varepsilon^2}\right)} + \mathcal{O}\left(\Delta x^2\right), \\
  C_j^N(F) =& \left(1+\frac{\sigma\Delta t}{\varepsilon^2}\left(1+\frac{\sigma \Delta t}{2\varepsilon^2}\right)\right)\partial_t F^{n+1/2}_j \\
           &+\frac{1}{\varepsilon}\left(\partial_x E^{n+1/2}_j+\frac{\sigma\Delta t}{2\varepsilon^2}\partial_x E^{n+1}_j\right)+\frac{\sigma}{\varepsilon^2}\left(1+\frac{\sigma \Delta t}{2\varepsilon^2}\right)F_j^n
 \end{align*}
 
 Using the identities
      ${\partial_t F^{n+1/2}_j}+(1/\varepsilon){\partial_x E^{n+1/2}_j}=-(\sigma/\varepsilon^2){F^{n+1/2}_j}$
   and $F^n_j+\Delta t \partial_t F^{n+1/2}_j = F^{n+1}_j +\mathcal{O}\left(\Delta t^3\right)$ leads to:
   \begin{equation*}
   C_j^N(F) = \frac{\sigma}{\varepsilon^2}\left(\left(1+\frac{\sigma \Delta t}{2\varepsilon^2}\right)F^{n+1}_j+\frac{\Delta t}{2\varepsilon}\partial_xE^{n+1}_j-F^{n+1/2}_j + \mathcal{O}\left(\left(1+\frac{\sigma \Delta t}{2\varepsilon^2}\right)\Delta t^3\right)\right).
   \end{equation*}
   Now, using $\partial_xE^{n+1}_j+(\sigma/\varepsilon)F^{n+1}_j = \partial_{t}F^{n+1}_j$ provides
   \begin{align*}
   C_j^N(F) &= \frac{\sigma}{\varepsilon^2}\left(F^{n+1}_j-F^{n+1/2}_j - \frac{\Delta t}{2}\partial_t F^{n+1}_j + \mathcal{O}\left(\left(1+\frac{\sigma \Delta t}{2\varepsilon^2}\right)\Delta t^3\right)\right) \\
   &= \mathcal{O}\left(\frac{\sigma\Delta t^2}{\varepsilon^2} \right)+ \mathcal{O}\left(\frac{\sigma}{\varepsilon^2}\left(1+\frac{\sigma \Delta t}{2\varepsilon^2}\right)\Delta t^3\right) = \mathcal{O}(\Delta x^2)
   \end{align*}
   using the previous estimates.
   
  Concerning $E$, we have:
  \begin{equation*}
    c_j^n(E)= \frac{E^{n+1}_j-E^n_j}{\Delta t}+\frac{M_1}{\varepsilon}\frac{F^n_{j+1}-F^n_{j-1}}{2\Delta x}-\frac{M_1^+\Delta t}{2\varepsilon^2}\frac{E^n_{j+1}-2E^n_j+E^n_{j-1}}{\Delta x^2}.
  \end{equation*}
  First, the flux discretization provides:
  \begin{equation*}
      \frac{M_1}{\varepsilon}\frac{F^n_{j+1}-F^n_{j-1}}{2\Delta x}=\frac{M_1}{\varepsilon}\partial_x F^n_j +\frac{M_1\Delta x^2}{2\varepsilon}\partial_{xxx}F^n_j +\mathcal{O}\left(\frac{M_1}{\varepsilon}\Delta x^4\right),
  \end{equation*}
  with $\partial_{xxx}F^n_j=\mathcal{O}\left(\varepsilon\right)$ again such that the error is $\mathcal{O}\left(\Delta x^2\right)$. Second, the diffusion term provides:
  \begin{equation*}
      \frac{M_1^+\Delta t}{2\varepsilon^2}\frac{E^n_{j+1}-2E^n_j+E^n_{j-1}}{\Delta x^2}=\frac{M_1^+\Delta t}{2\varepsilon^2}\partial_{xx}E^n_j+\mathcal{O}\left(\frac{M_1^+\Delta t}{2\varepsilon^2}\Delta x^2\right),
  \end{equation*}
  but
    ${M_1^+\Delta t}/({2\varepsilon^2})
    \leq 1/{\sigma}=\mathcal{O}\left(1\right)$  so that $\mathcal{O}\left({M_1^+\Delta t}\Delta x^2 / (2\varepsilon^2)\right)=\mathcal{O}\left(\Delta x^2\right).$
  Under the usual stability conditions, we have $\mathcal{O}\left(\Delta t^2\right)=\mathcal{O}\left(\Delta x^2\right)$. Now we have:
  \begin{equation*}
    c_j^n(E)
    {=}\partial_t E^{n+1/2}_j+\frac{M_1}{\varepsilon}\partial_xF^n_j-\frac{M_1^+\Delta t}{2\varepsilon^2}\partial_{xx}E^n_j + \mathcal{O}\left(\Delta x^2\right),
  \end{equation*}
  which leads to:
  \begin{align*}
    c_j^n(E) &= \frac{C_j^n(E)}{1+\frac{\sigma \Delta t}{2\varepsilon^2}\left(1+\frac{\sigma \Delta t}{2\varepsilon^2}\right)} + \mathcal{O}\left(\Delta x^2\right)\\
    C_j^n(E) &= \bigg(\left(1+\frac{\sigma\Delta t}{2\varepsilon^2}\left(1+\frac{\sigma\Delta t}{2\varepsilon^2}\right)\right)\partial_tE^{n+1/2}_j +\frac{1}{\varepsilon}\partial_xF^n_j-\frac{\Delta t}{2\varepsilon^2}\left(1+\frac{\sigma\Delta t}{2\varepsilon^2}\right)\partial_{xx}E^n_j\bigg)\\
   &=\frac{-\Delta t}{2\varepsilon}\Bigg(\frac{\partial_x F^{n+1/2}-\partial_x F^n_j}{\Delta t/2}+\frac{1}{\varepsilon}\partial_{xx}E^n_j
   \\ &\qquad\qquad+\frac{\sigma}{\varepsilon^2}\left(\partial_x F^{n+1/2}_j+\frac{\Delta t}{2\varepsilon}\left(\partial_{xx}E^n_j+\frac{\sigma}{\varepsilon}\partial_xF^{n+1/2}_j\right)\right)\Bigg).
  \end{align*}
  Using the usual Taylor expansions we obtain in the end:
  \begin{align*}
      c_j^n(E)
      {=}&\mathcal{O}\left(\frac{\Delta t^2}{\varepsilon\left(1+\frac{\sigma\Delta t}{2\varepsilon^2}\right)}\right)+\mathcal{O}\left(\frac{\Delta t^3}{\varepsilon^4\left(1+\frac{\sigma\Delta t}{2\varepsilon^2}\left(1+\frac{\sigma \Delta t}{2\varepsilon^2}\right)\right)}\right) + \mathcal{O}(\Delta x^2),
  \end{align*}
and the estimations found for $c_j^n(F)$ leads to simplify $c_j^n(E) = \mathcal{O}(\Delta x^2)$.

\section{Alternative discretizations for order 1 ImEx schemes: upwind discretization}
\label{supsec:HHEord1FV}
Although we have shown that scheme \cref{eq:AP1_ctr_HHE} demonstrates strong robustness property under the right set of stability conditions \cref{cnd:LInftyStbltyHHEord1}, the presence of a lower bound to ensure that shock do not trigger spurious oscillations can be seen as too restrictive.
 An attempt to remove this lower bound constraint can be made in using an upwind discretization of the fluxes, leading to the ImEx1-upwd scheme:
\begin{subequations}\label{eq:AP1_dctr_HHE}
  \begin{align}
    \frac{E^{n+1}_j - E^n_j}{\Delta t}&+\frac{M}{\varepsilon}\frac{F^n_{j+1}-F^n_{j-1}}{2\Delta x}\\
    &-\left( \frac{M\Delta x}{2\varepsilon} + \frac{M\Delta t}{\varepsilon^2} \right)\frac{E^n_{j+1}-2E^n_{j}+E^n_{j-1}}{\Delta x^2} =0, \nonumber\\ 
    \frac{F^{n+1}_j - F^n_j}{\Delta t}&+\frac{M}{\varepsilon}\frac{E^n_{j+1}-E^n_{j-1}}{2\Delta x}\\
    &-\left( \frac{M\Delta x}{2\varepsilon} + \frac{M\Delta t}{\varepsilon^2} \right)\frac{F^n_{j+1}-2F^n_{j}+F^n_{j-1}}{\Delta x^2} =-M\frac{\sigma}{\varepsilon^2}F^n_j. \nonumber
  \end{align}
\end{subequations}
\begin{theorem}\label{thm:LinftyStabAP1dctr}
  Assuming that periodic \cref{bc:periodic} or hybrid \cref{bc:hybrid} boundary conditions are used, the ImEx1-upwd scheme \cref{eq:AP1_dctr_HHE} is $l^{\infty}$-diminishing for variables $u,v$ under the condition:
  \begin{equation}\label{cnd:LInftyStbltyHHE_AP1dctr}
    \Delta t \leq \Delta t_{max} :=\frac{\sigma \Delta x^2}{4}\left(\left(\frac{1}{2}-\frac{\varepsilon}{\sigma \Delta x}\right)+\sqrt{\left(\frac{\varepsilon}{\sigma \Delta x}-\frac{1}{2}\right)^2+2\left(\frac{2\varepsilon}{\sigma \Delta x}\right)^2}\right).
  \end{equation}
\end{theorem}
\begin{proof}
  The proof is very similar to that of \cref{thm:LinftyStabAP1ctr}, using a form of the scheme \cref{eq:AP1_dctr_HHE} equivalent to the reordering \cref{eq:reorderedAP1_ctr_HHE_uv}
there is no longer any lower bound condition and the only stability restriction is
\begin{equation*}
  1-\frac{2M\Delta t^2}{\varepsilon^2\Delta x^2}-\frac{M\Delta t}{\varepsilon\Delta x}-\frac{M\sigma \Delta t}{2\varepsilon^2}\geq 0,
\end{equation*}
which leads to \cref{cnd:LInftyStbltyHHE_AP1dctr}
\end{proof}
\begin{theorem}
  In the regime $\Delta x \approx \varepsilon$, the ImEx1-upwd scheme \cref{eq:AP1_dctr_HHE} loses accuracy in the sense $c_j^n(E) = \mathcal{O}(1)$.
\end{theorem}
\begin{proof}
  As compared to the consistency error conducted on scheme \cref{eq:AP1_ctr_HHE} we have an additional error term for variable $E$ of the form
    $({M\Delta x}/{2\varepsilon})\partial_{xx} E$.
  However in the regime $\varepsilon\approx\Delta x$, we have $M\approx 1$ and $\Delta x/\varepsilon \approx 1$ and consequently this new term is of order $\mathcal{O}\left( \Delta x^{0} \right)$ in this regime.
\end{proof}
\begin{remark}
  The impossibility to use the classical upwind discretization is not commonly discussed in literature but here we have shown that it does not yield an asymptotic preserving scheme.
\end{remark}

\section{Reverse Runge-Kutta framework}
\label{supsec:RRK}

The approach we have used in order to integrate in time the source term is essential in the construction of the time-space-ImEx method proposed in the paper. We have called it a Reverse Runge-Kutta method. It has been introduced in the classical monographs \cite{Butcher2016SM}, where it is referred to as {\sl reflected}, and \cite{Hairer1993SM}, under the denomination {\sl adjoint} from a classical explicit RK method. 
Even if some general properties of these this specific class of method are provided in these
works, the efficiency and computational aspects were rather investigated in \cite{Bokhoven80,Cash82SM} (where they are called implicit endpoint quadrature formulas and mono-implicite RK Formulae respectively), and 
more recently in \cite{Kalitkin14SM}  and \cite{Skvortsov17SM} (called inverse RK schemes and implicit RK methods obtained as a result of the inversion of explicit methods).

Before introducing the core principle of the approach, since it is essentially based on 
 explicit Runge-Kutta methods, we introduce first-order and second order RK methods and the related notations and Butcher arrays. 
 The general form of the autonomous ODE we want to solve reads:
\begin{equation*}
  \text{d}_{t} U  = f(U),
\end{equation*}
where $f$ is assumed to verify the assumptions of the Cauchy-Lipschitz theorem. 
With these notations the explicit Euler method and its corresponding Butcher array read:
\begin{equation*}
  \begin{array}{c|c}
    0 &  \\
    \hline
      & 1
  \end{array}
  \quad\quad
  U_{n+1} = U_{n}+\Delta t f(U_{n}).
\end{equation*}
The general form of explicit second-order methods parametrized by $\alpha \in \left(0,1\right]$ reads:
\begin{equation*}
 \begin{array}{c|cc}
    0 &   &  \\
    \alpha & \alpha & \\
    \hline
     \vphantom{\Bigg(}& \displaystyle1-\frac{1}{2\alpha} & \displaystyle \frac{1}{2\alpha}
  \end{array}
\quad\quad 
  \begin{aligned}
      U_{n_{1}} &= U_{n}+ \alpha \Delta t f\left( U_{n} \right),\\
        U_{n+1} &= U_{n}+\Delta t \left( 1-\frac{1}{2\alpha} \right)f\left( U_{n} \right)+\frac{\Delta t}{2\alpha}f\left( U_{n_{1}} \right).\\
  \end{aligned}
 \end{equation*}
For $\alpha = 1$ we retrieve the Heun method and for $\alpha = 1/2$, the classical RK2:
\begin{equation*}
   \begin{aligned}
    U_{n_{1}} &= U_{n} + \frac{\Delta t}{2} f\left( U_{n} \right),\\
     U_{n+1}   &= U_{n} + \Delta t f\left( U_{n_{1}} \right).\\
    \end{aligned}
\end{equation*}

The principle of the Reverse Runge-Kutta method is the following: we substitute $-t$ to $t$ (reversed time) and we apply the explicit method starting from the final point.
For instance, at order one and using the explicit Euler method we obtain:
\begin{align*}
  U_{n} = U_{n+1}-\Delta t f\left( U_{n+1} \right), 
\end{align*}
which boils down to the  backward Euler method:
\begin{align*}
U_{n+1} = U_{n}+\Delta t f\left( U_{n+1} \right).
\end{align*}
At second order, we obtain:
\begin{align*}
  U_{n} &= U_{n+1}-\Delta t \left( 1-\frac{1}{2\alpha} \right) f\left( U_{n+1} \right)-\frac{\Delta t}{2\alpha} f\Big( U_{n+1}-\alpha \Delta t f(U_{n+1}) \Big), 
 \end{align*}
which can be rewritten as:
\begin{align*}
U_{n+1} &= U_{n} +\Delta t \left( 1-\frac{1}{2\alpha} \right) f\left( U_{n+1} \right)+\frac{\Delta t}{2\alpha} f\Big( U_{n+1}-\alpha \Delta t f(U_{n+1}) \Big).
\end{align*}

Regarding absolute stability, the general stability function of an explicit second-order method is
\begin{equation*}
  R\left( z \right) = 1+z+\frac{z^{2}}{2}.
\end{equation*}
Consequently, the stability function $R^{*}(z)$ of second order Reverse RK methods can be obtained by substituting $-z \to z$ and taking the inverse of $R$, leading to:
\begin{equation*}
  R^{*}\left( z \right) = \frac{1}{R(-z)}.
\end{equation*}
Because $1/\left|1-z+z^{2}/2 \right|\leq 1$ for all $z\in \mathbb{C}^{-}$ the second-order Reverse RK methods are $A$-stable and even $L$-stable.

These Reverse RK methods can be recast within the framework of classical RK methods (the Butcher array can be obtained explicitly from the original one  \cite{Hairer1993SM}):
\begin{equation*}
  \left|
  \begin{aligned}
    U_{n_{1}} &= U_{n}+\frac{\Delta t}{2\alpha} f\left( U_{n_{1}} \right)+ \Delta t \left( 1-\frac{1}{2\alpha} \right) f\left( U_{n_{2}} \right)-\alpha \Delta t f\left( U_{n_{2}} \right),\\
    U_{n_{2}} &= U_{n}+\frac{\Delta t}{2\alpha} f\left( U_{n_{1}} \right)+ \Delta t \left( 1-\frac{1}{2\alpha} \right) f\left( U_{n_{2}} \right),\\
    U_{n+1} &= U_{n}+\frac{\Delta t}{2\alpha} f\left( U_{n_{1}} \right)+ \Delta t \left( 1-\frac{1}{2\alpha} \right) f\left( U_{n_{2}} \right).
  \end{aligned}  
  \right. \quad
\end{equation*}
One notices that $U_{n_{1}} = U_{n_{2}}-\alpha \Delta t f\left( U_{n_{2}} \right)$, so that it suffices to solve for $U_{n_{2}}$ directly. The corresponding Butcher arrays for $\alpha \in \left(0,1\right]$ and $\alpha=1/2$ (Reverse RK2 method used in the paper) read:
\begin{equation*}
  \begin{array}{c|cc}
    \vphantom{\Bigg(}1-\alpha\vphantom{} & \displaystyle \frac{1}{2\alpha}  & \displaystyle 1-\frac{1}{2\alpha}-\alpha \\
    \vphantom{\Bigg(}1 & \displaystyle \frac{1}{2\alpha} & \displaystyle 1-\frac{1}{2\alpha}\\
    \hline
    \ \vphantom{\Bigg(}& \displaystyle \frac{1}{2\alpha} &  \displaystyle 1-\frac{1}{2\alpha}
  \end{array}
  \qquad\qquad
  \begin{array}{c|cc}
    \displaystyle\frac{1}{2} &  \vphantom{\bigg(}1 &\displaystyle -\frac{1}{2} \\
    1 & 1 &  \ \ \vphantom{\Bigl(} 0\\
    \hline
     \ \vphantom{\Bigl(} & 1 & \ \ 0
  \end{array}
\end{equation*}
which shows that the Reverse RK methods are stiffly accurate and have excellent stability properties. They can be shown to be computationally efficient as mono-implicit methods.

\section{Riemann problem, second order and $l^\infty$-stability without limitation} \label{supsec:RP}
We consider the in this section the same Riemann-problem initial condition as in \Cref{subsec:numRP}, which reads in the linear case:
\begin{equation*}
  E(0,x)  = \mathbf{1}_{\left\{ x\leq \frac{x_R+x_L}{2} \right\}} E_L + \mathbf{1}_{\left\{ x> \frac{x_R+x_L}{2} \right\}} E_R, \quad
  F(0,x ) = 0,
\end{equation*}
and in the nonlinear case
\begin{equation*}
  \rho(0,x)  = \mathbf{1}_{\left\{ x\leq \frac{x_R+x_L}{2} \right\}} \rho_L + \mathbf{1}_{\left\{ x> \frac{x_R+x_L}{2} \right\}} \rho_R, \quad
  \left( \rho u \right)(0,x ) = 0.
\end{equation*}

We show that for both the linear case in the hyperbolic regime (meaning $\varepsilon = 0.5$) and the nonlinear case in the intermediate parabolic-hyperbolic regime (meaning $\varepsilon = 5.10^{-2}$ and $\varepsilon = 10^{-2}$), the second-order centered methods do not trigger spurious oscillations around discontinuities without using additional numerical diffusion or slope limiters, provided that the time step falls within the acceptable range of values determined in \Cref{cnd:LInftyStbltyHHE_AP2ctr} implying $l^\infty$-stability. Within this regimes, the methods are quite accurate. 

\begin{figure}[htbp]
  \centering
  \includegraphics[width=1.\columnwidth]{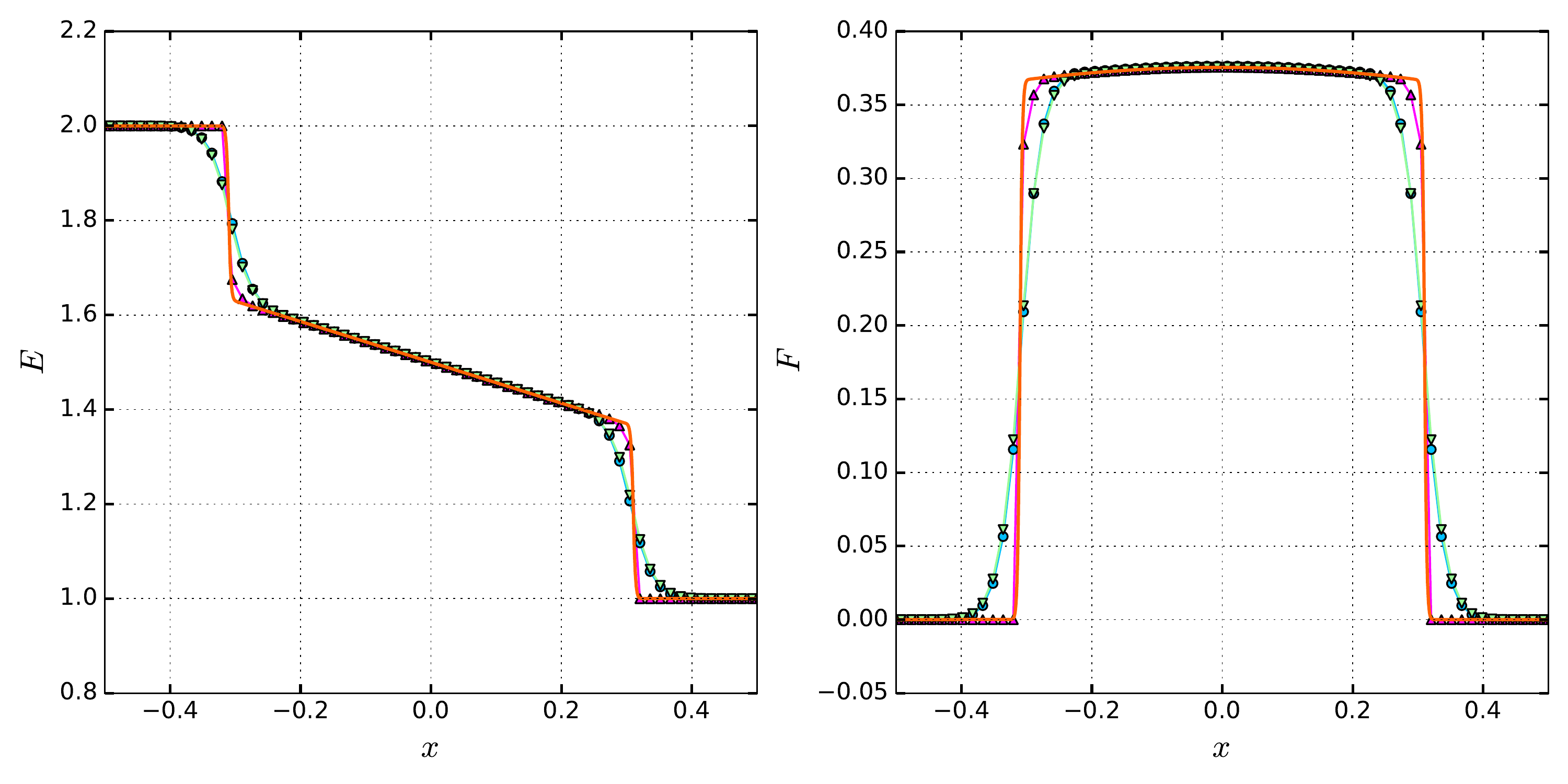}
  \caption{Variables $E$ (left) and $F$ (right) obtained at time $t_{f} = 20\Delta t_{ctr}\approx 0.15$, for $\varepsilon = 0.5$, $E_{L} = 2$ and $E_{R} = 1$, respectively with the ImEx2-ctr method ($\triangle$), $\Delta t_{ctr} = \left( \Delta t_{max}+\Delta t_{min} \right)/2$, from \cref{cnd:LInftyStbltyHHE_AP2ctr1,cnd:LInftyStbltyHHE_AP2ctr2}, ImEx2-minmod method ($\circ$), $\Delta t_{mnmd} = 1.2 \Delta t_{max}$, from \cref{cnd:L2stbltyHHEord2} and MUSCL-Hancock method with Strang splitting, Reverse RK2 ($\triangledown$), $\Delta t_{MH} = 0.9 \min\left( 2\varepsilon^{2}/\sigma_{max},\varepsilon \Delta x/\left( u_{max}+c \right) \right)$, using $N=64$ cells, with the latter scheme used as reference (solid line) with $N=2048$ cells.}
  \label{fig:ShockHHE2ndOrdComp}
\end{figure}

\begin{figure}[htbp]
  \centering
  \includegraphics[width=1.\columnwidth]{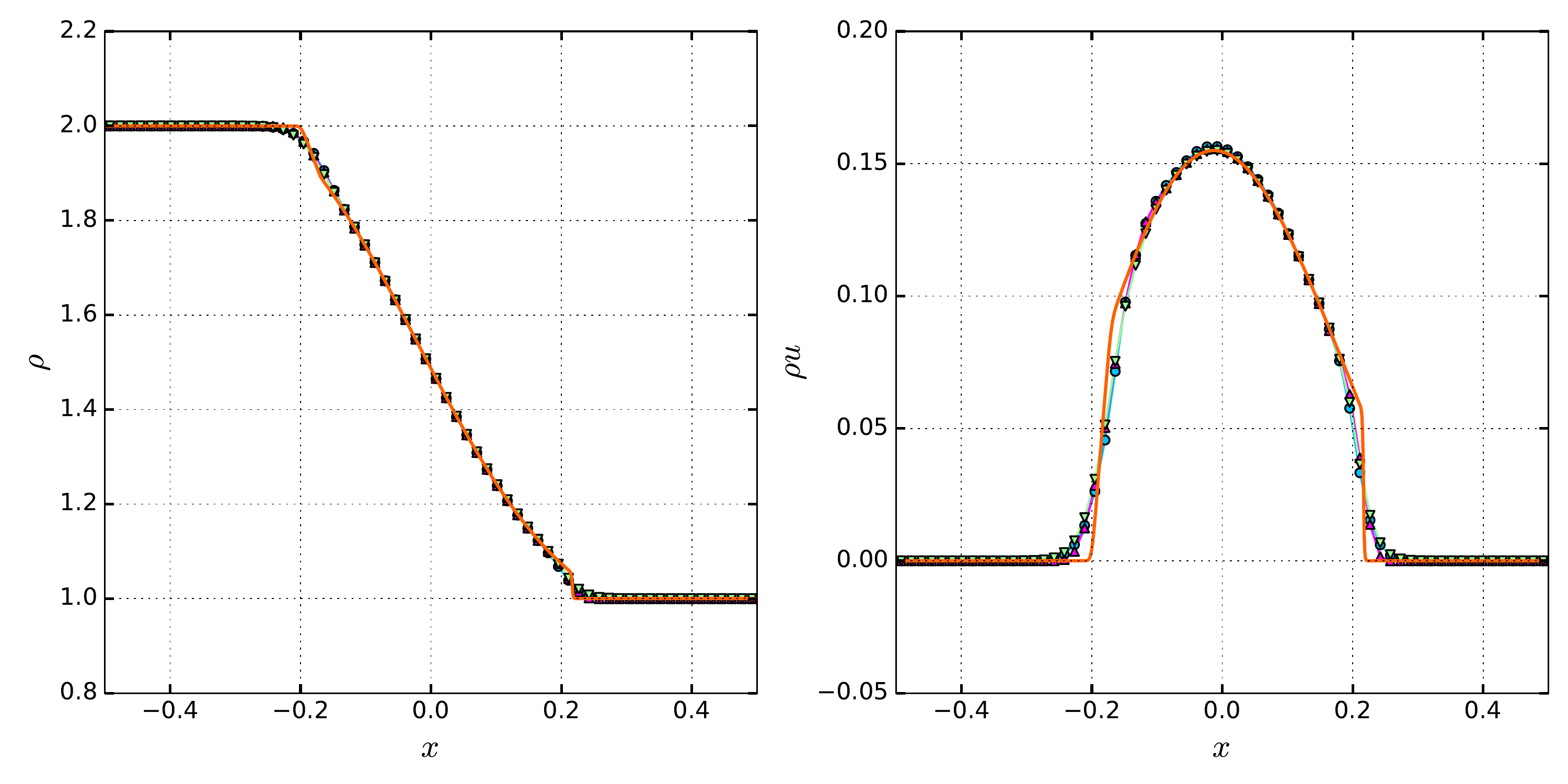}
  \caption{Mass $\rho$ and momentum $\rho u$ obtained at time $t_{f} = 0.01$ for $\varepsilon = 5\times 10^{-2}$, $\rho_{L} = 2$ and $\rho_{R} = 1$, respectively with the ImEx2-ctr method ($\triangle$), $\Delta t_{ctr} = 1.87\Delta t_{max}$, from \cref{cnd:L2stbltyHHEord2}, ImEx2-minmod method ($\circ$), $\Delta t_{mnmd} = 0.9 \Delta t_{max}$, from \cref{cnd:L2stbltyHHEord2} and MUSCL-Hancock method with Strang splitting, Reverse RK2 ($\triangledown$), $\Delta t_{MH} = 0.9 \min\left( 2\varepsilon^{2}/\sigma_{max},\varepsilon \Delta x/\left( u_{max}+c \right) \right)$, using $N=64$ cells, with the latter scheme used as reference (solid line) with $N=2048$ cells.}
  \label{fig:ShockEF2ndOrdInterRegimeComp}
\end{figure}

\begin{figure}[htbp]
  \centering
  \includegraphics[width=1.\columnwidth]{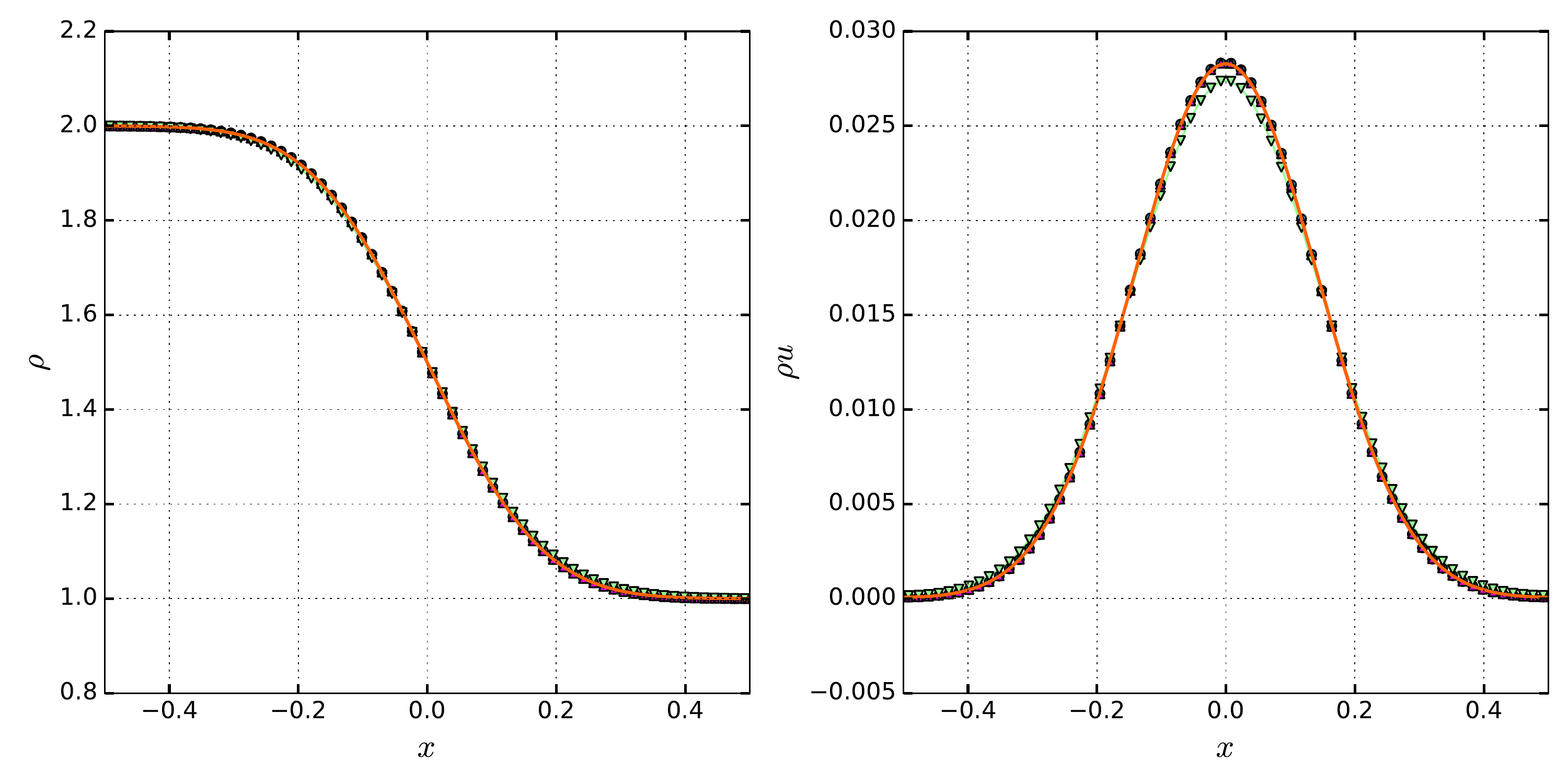}
  \caption{Mass $\rho$ and momentum $\rho u$ obtained at time $t_{f} = 0.01$ for $\varepsilon = 10^{-2}$, $\rho_{L} = 2$ and $\rho_{R} = 1$, respectively with the ImEx2-ctr method ($\triangle$), $\Delta t_{ctr} = 0.9\Delta t_{max}$, from \cref{cnd:L2stbltyHHEord2}, ImEx2-minmod method ($\circ$), $\Delta t_{mnmd} = 0.9 \Delta t_{max}$, from \cref{cnd:L2stbltyHHEord2} and MUSCL-Hancock method with Strang splitting, Reverse RK2 ($\triangledown$), $\Delta t_{MH} = 0.9 \min\left( 2\varepsilon^{2}/\sigma_{max},\varepsilon \Delta x/\left( u_{max}+c \right) \right)$, using $N=64$ cells, with the latter scheme used as reference (solid line) with $N=2048$.}
  \label{fig:ShockEF2ndOrdDiffRegimeComp}
\end{figure}

For the linear hyperbolic regime case, shown in \Cref{fig:ShockHHE2ndOrdComp}, one observes that the ImEx2-ctr method does not generate oscillations in the vicinity of the discontinuity but, as expected, it is also significantly less diffusive than the ImEx2-minmod and MUSCL-Hancock methods, which both achieve similar performances.
We emphasize however that in this case the estimations for $\Delta t_{min}$ and $\Delta t_{max}$ provided by \cref{cnd:LInftyStbltyHHE_AP2ctr} collide and therefore one must compute numerically the bound $\Delta t_{min}$ and $\Delta t_{max}$ corresponding to conditions \cref{cnd:LInftyStbltyHHE_AP2ctr1,cnd:LInftyStbltyHHE_AP2ctr2} and choose $\Delta t$ within the proper interval $\left[ \Delta t_{min}, \Delta t_{max} \right]$.

It is remarkable to observe that unlike in the case of purely hyperbolic equations, for which second order methods always require slope limiters in order to be stable around shocks, the presence of a source term guarantees that there is always an interval of possible time steps for which the ImEx2-ctr method can be $l^\infty$-stable.

Regarding the nonlinear intermediate and diffusive regimes case ($\varepsilon = 5\times 10^{-2}$ and $\varepsilon = 10^{-2}$), the interval of admissible time steps $[\Delta t_{min}, \Delta t_{max}]$ is larger than in the hyperbolic regime and as a result the ImEx2-ctr method does not generate any oscillations, as can be seen in \Cref{fig:ShockEF2ndOrdInterRegimeComp,fig:ShockEF2ndOrdDiffRegimeComp}. We can also note that the ImEx2-minmod method yields equivalent performances to that of the ImEx2-ctr method, as expected from the theoretical study of the paper, while the classic MUSCL-Hancock method with splitting is already lagging behind in terms of accuracy.

\section*{Acknowledgments}
We acknowledge the precious help of Lo\"ic Gouarin and the use of the samurai code he develops  (\url{https://github.com/hpc-maths/samurai}) within the framework of Initiative HPC@Maths (PI M. Massot and L. Gouarin - \url{https://initiative-hpc-maths.gitlab.labos.polytechnique.fr/site/}).

\bibliographystyle{siamplain}
\bibliography{Reboul_Pichard_Massot_ImEx_2022}

\begin{thebibliography}{10}

\bibitem{albi2020implicit}
{\sc G.~Albi, G.~Dimarco, and L.~Pareschi}, {\em Implicit-explicit multistep
  methods for hyperbolic systems with multiscale relaxation}, SIAM J. Sci.
  Comput., 42 (2020), pp.~A2402--A2435.

\bibitem{alvarezlaguna20}
{\sc A.~Alvarez-Laguna, T.~E. Magin, M.~Massot, A.~Bourdon, and P.~Chabert},
  {\em {Plasma-sheath transition in multi-fluid models with inertial terms
  under low pressure conditions: Comparison with the classical and kinetic
  theory}}, {Plasma Sources Science and Technology}, 29 (2020), p.~025003.

\bibitem{alvarezlaguna20jcp}
{\sc A.~Alvarez-Laguna, T.~Pichard, T.~Magin, P.~Chabert, A.~Bourdon, and
  M.~Massot}, {\em {An asymptotic preserving well-balanced scheme for the
  isothermal fluid equations in low-temperature plasma applications}}, {J.
  Comput. Phys.}, 419 (2020), p.~109634.

\bibitem{Bardos-Golse-Levermore-I91}
{\sc C.~Bardos, F.~Golse, and D.~Levermore}, {\em Fluid dynamic limits of
  kinetic equations. {I.} {F}ormal derivations}, J. Stat. Phys., 63 (1991),
  pp.~323--344.

\bibitem{berthon2011asymptotic}
{\sc C.~Berthon and R.~Turpault}, {\em Asymptotic preserving {HLL} schemes},
  Numer. Methods Partial Differ. Equ., 27 (2011), pp.~1396--1422.

\bibitem{Boscarino-LeFloch-Russo14}
{\sc S.~Boscarino, P.~G. LeFloch, and G.~Russo}, {\em High-order
  asymptotic-preserving methods for fully nonlinear relaxation problems}, SIAM
  J. Sci. Comput., 36 (2014), pp.~A377--A395.

\bibitem{boscarino2013implicit}
{\sc S.~Boscarino, L.~Pareschi, and G.~Russo}, {\em Implicit-explicit
  {Runge-Kutta} schemes for hyperbolic systems and kinetic equations in the
  diffusion limit}, SIAM J. Sci. Comput., 35 (2013), pp.~A22--A51.

\bibitem{Boscarino-Russo13}
{\sc S.~Boscarino and G.~Russo}, {\em Flux-explicit {IMEX Runge-Kutta} schemes
  for hyperbolic to parabolic relaxation problems}, SIAM J. Numer. Anal., 51
  (2013), pp.~163--190.

\bibitem{bouchut2004nonlinear}
{\sc F.~Bouchut}, {\em Nonlinear stability of finite Volume Methods for
  hyperbolic conservation laws: And Well-Balanced schemes for sources},
  Springer Science \& Business Media, 2004.

\bibitem{buet2012design}
{\sc C.~Buet, B.~Despr{\'e}s, and E.~Franck}, {\em Design of asymptotic
  preserving finite volume schemes for the hyperbolic heat equation on
  unstructured meshes}, Numer. Math., 122 (2012), pp.~227--278.

\bibitem{Butcher2016}
{\sc J.~C. Butcher}, {\em Numerical methods for ordinary differential
  equations}, John Wiley \& Sons, Ltd., Chichester, third~ed., 2016.
\newblock With a foreword by J. M. Sanz-Serna.

\bibitem{Butcher2016SM}
{\sc J.~C. Butcher}, {\em Numerical methods for ordinary differential
  equations}, John Wiley \& Sons, Ltd., Chichester, third~ed., 2016.
\newblock With a foreword by J. M. Sanz-Serna.

\bibitem{Caflisch-Jin-Russo97}
{\sc R.~E. Caflisch, S.~Jin, and G.~Russo}, {\em Uniformly accurate schemes for
  hyperbolic systems with relaxation}, SIAM J. Numer. Anal., 34 (1997),
  pp.~246--281.

\bibitem{Cash82SM}
{\sc J.~R. Cash and A.~Singhal}, {\em {Mono-implicit Runge-Kutta Formulae for
  the Numerical Integration of Stiff Differential Systems}}, IMA Journal of
  Numerical Analysis, 2 (1982), pp.~211--227.

\bibitem{Cat48}
{\sc C.~Cattaneo}, {\em Sulla conduzione del calore}, Atti del Seminario
  Matematico e Fisico dell{'}Universit\`a di {M}odena, 3 (1948), pp.~3--21.

\bibitem{Cat58}
{\sc C.~Cattaneo}, {\em Sur une forme de l'\'equation de la chaleur \'eliminant
  le paradoxe d'une propagation instantan\'ee}, C. R. Acad. Sci. Paris, 247
  (1958), pp.~431--433.

\bibitem{chalons2019high}
{\sc C.~Chalons and R.~Turpault}, {\em High-order asymptotic-preserving schemes
  for linear systems: {A}pplication to the {Goldstein-Taylor} equations},
  Numer. Methods Partial Differ. Equ., 35 (2019), pp.~1538--1561.

\bibitem{Chen-Lervemore-Liu94}
{\sc G.-Q. Chen, C.~D. Levermore, and T.-P. Liu}, {\em Hyperbolic conservation
  laws with stiff relaxation terms and entropy}, Commun. Pure Appl. Math., 47
  (1994), pp.~787--830.

\bibitem{godlewski2013numerical}
{\sc E.~Godlewski and P.-A. Raviart}, {\em Numerical approximation of
  hyperbolic systems of conservation laws}, vol.~118, Springer Science \&
  Business Media, 2013.

\bibitem{gosse2002asymptotic}
{\sc L.~Gosse and G.~Toscani}, {\em An asymptotic-preserving well-balanced
  scheme for the hyperbolic heat equations}, C. R. Acad. Sci. Paris, 334
  (2002), pp.~337--342.

\bibitem{Hairer1993}
{\sc E.~Hairer, S.~P. N{\o}rsett, and G.~Wanner}, {\em Solving ordinary
  differential equations. {I}}, vol.~8 of Springer Series in Computational
  Mathematics, Springer-Verlag, Berlin, second~ed., 1993.
\newblock Nonstiff problems.

\bibitem{Hairer1993SM}
{\sc E.~Hairer, S.~P. N{\o}rsett, and G.~Wanner}, {\em Solving ordinary
  differential equations. {I}}, vol.~8 of Springer Series in Computational
  Mathematics, Springer-Verlag, Berlin, second~ed., 1993.
\newblock Nonstiff problems.

\bibitem{Jin11}
{\sc S.~Jin}, {\em Asymptotic preserving ({AP}) schemes for multiscale kinetic
  and hyperbolic equations: a review}, Rivista di Matematica della Universita
  di Parma, 3 (2012), pp.~177--216.

\bibitem{Jin-Levermore96}
{\sc S.~Jin and C.~D. Levermore}, {\em Numerical schemes for hyperbolic
  conservation laws with stiff relaxation terms}, J. Comput. Phys., 126 (1996),
  pp.~449--467.

\bibitem{Jin-Pareschi-Toscani98}
{\sc S.~Jin, L.~Pareschi, and G.~Toscani}, {\em Diffusive relaxation schemes
  for multiscale discrete-velocity kinetic equations}, SIAM J. Numer. Anal., 35
  (1998), pp.~2405--2439.

\bibitem{Jin-Pareschi_Toscani00}
{\sc S.~Jin, L.~Pareschi, and G.~Toscani}, {\em Uniformly accurate diffusive
  relaxation schemes for multiscale transport equations}, SIAM J. Numer. Anal.,
  38 (2000), pp.~913--936.

\bibitem{Kalitkin14}
{\sc N.~N. Kalitkin and I.~P. Poshivaylo}, {\em Computations with inverse
  {Runge}-{Kutta} schemes}, Mathematical Models and Computer Simulations, 6
  (2014), pp.~272--285.

\bibitem{Kalitkin14SM}
{\sc N.~N. Kalitkin and I.~P. Poshivaylo}, {\em Computations with inverse
  {Runge}-{Kutta} schemes}, Mathematical Models and Computer Simulations, 6
  (2014), pp.~272--285.

\bibitem{Klar98}
{\sc A.~Klar}, {\em An asymptotic-induced scheme for nonstationary transport
  equations in the diffusive limit}, SIAM J. Numer. Anal., 35 (1998),
  pp.~1073--1094.

\bibitem{Lax_thm}
{\sc P.~D. Lax and R.~D. Richtmyer}, {\em Survey of the stability of linear
  finite difference equations}, Comm. Pure Appl. Math., 9 (1956), p.~267–293.

\bibitem{Lax_Wendroff}
{\sc P.~D. Lax and B.~Wendroff}, {\em Systems of conservation laws}, Comm. Pure
  Appl. Math., 13 (1960), p.~217–237.

\bibitem{Max67}
{\sc J.~C. Maxwell}, {\em On the dynamical theory of gases}, Philos. Trans. A:
  Math. Phys. Eng. Sci., 157 (1867), pp.~49--88.

\bibitem{Natalini96}
{\sc R.~Natalini}, {\em Convergence to equilibrium for the relaxation
  approximations of conservation laws}, Commun. Pure Appl. Math., 49 (1996),
  pp.~795--823.

\bibitem{Reboul23}
{\sc L.~Reboul}, {\em Mathematical modeling and simulation of non-equilibrium
  plasmas for the prediction of the electric propulsion}, PhD thesis, EDMH,
  Institut Polytechnique de Paris, 2023.

\bibitem{Reboul2021VKI}
{\sc L.~Reboul, A.~A. Laguna, T.~Magin, A.~Bourdon, and M.~Massot}, {\em Fluid
  modeling of low-temperature partially-magnetized plasmas in the presence of
  sheaths: numerical issues for the simulation of instabilities and comparison
  with kinetic simulations}, Physics of Plasmas,  (2022).
\newblock To be submitted (a preliminary version has appeared as a preprint of
  the von {K}arman PhD symposium in 2021).

\bibitem{Skvortsov17}
{\sc L.~M. Skvortsov}, {\em On implicit {Runge}–{Kutta} methods obtained as a
  result of the inversion of explicit methods}, Mathematical Models and
  Computer Simulations, 9 (2017), pp.~498--510.

\bibitem{Skvortsov17SM}
{\sc L.~M. Skvortsov}, {\em On implicit {Runge}–{Kutta} methods obtained as a
  result of the inversion of explicit methods}, Mathematical Models and
  Computer Simulations, 9 (2017), pp.~498--510.

\bibitem{toro2013riemann}
{\sc E.~F. Toro}, {\em Riemann solvers and numerical methods for fluid
  dynamics: a practical introduction}, Springer Science \& Business Media,
  2013.

\bibitem{Bokhoven80}
{\sc W.~M.~G. van Bokhoven}, {\em Efficient higher order implicit one-step
  methods for integration of stiff differential equations}, BIT Numerical
  Mathematics, 20 (1980), pp.~34--43.

\bibitem{Ver58}
{\sc P.~Vernotte}, {\em Les paradoxes de la th\'eorie continue de l'\'equation
  de la chaleur}, C. R. Acad. Sci. Paris, 246 (1958), pp.~3154--3155.
\newblock and 1948 - Volume 227 - pages 43 and 114.

\bibitem{Zhang-Tabarrok99}
{\sc Y.~Zhang and B.~Tabarrok}, {\em Modifications to the {Lax–Wendroff}
  scheme for hyperbolic systems with source terms}, Int. J. Numer. Meth. Eng.,
  44 (1999), pp.~27--40.

\end{thebibliography}
\end{document}